\documentclass{amsart}

\usepackage[sort]{cite}

\usepackage[utf8]{inputenc}
\usepackage[english]{babel}

\usepackage{amssymb,latexsym,amscd,amsmath}
\usepackage{amsfonts}
\usepackage{array}
\usepackage{booktabs}
\usepackage{multirow}
\usepackage[section]{algorithm}

\newcommand{\refthm}[1]{{Theorem \ref{#1}}}
\newcommand{\refeqn}[1]{{(\ref{#1})}}

\newcommand{\refalg}[1]{Algorithm \ref{#1}}
\newcommand{\refsec}[1]{{section \ref{#1}}}

\newcommand{\reftab}[1]{{Table \ref{#1}}}
\newcommand{\refprop}[1]{{Proposition \ref{#1}}}
\newcommand{\reflem}[1]{{Lemma \ref{#1}}}

\newtheorem{theorem}{Theorem}[section]
\newtheorem{lemma}[theorem]{Lemma}
\newtheorem{proposition}[theorem]{Proposition}
\newtheorem{corollary}[theorem]{Corollary}
\theoremstyle{remark}
\newtheorem{remark}[theorem]{Remark}
\newtheorem{example}[theorem]{Example}

\newcommand{\C}{\mathbb{C}}
\newcommand{\Z}{\mathbb{Z}}
\newcommand{\Pj}{\mathbb{P}}
\newcommand{\R}{\mathbb{R}}
\newcommand{\F}{\mathbb{F}}

\newcommand{\tensor}[1]{\mathfrak{#1}}
\newcommand{\Var}[1]{\mathcal{#1}}
\newcommand{\vect}[1]{\mathbf{#1}}
\newcommand{\sten}[3]{\vect{#1}_{#2}^{#3}}
\newcommand{\Sec}[2]{\sigma_{#1}({#2})}
\newcommand{\Tang}[2]{\mathrm{T}_{#1} {#2}}
\newcommand{\Plane}[1]{\mathrm{#1}}

\newcommand{\ktimes}{\otimes} 
\newcommand{\ttimes}{\times}  

\title[An algorithm for generic and specific identifiability]{An algorithm for generic and low-rank specific identifiability of complex tensors}

\author{Luca Chiantini, Giorgio Ottaviani,  Nick Vannieuwenhoven}

\begin{document}

\begin{abstract}
We propose a new sufficient condition for verifying whether generic rank-$r$ complex tensors of arbitrary order admit a unique decomposition as a linear combination of rank-$1$ tensors. A practical algorithm is proposed for verifying this condition, with which it was established that in all spaces of dimension less than $15000$, with a few known exceptions,  listed in the paper, generic identifiability holds for ranks up to one less than the generic rank of the space. This is the largest possible rank value for which generic identifiability can hold, except for spaces with a perfect shape. The algorithm can also verify the identifiability of a given specific rank-$r$ decomposition, provided that it can be shown to correspond to a nonsingular point of the $r$th order secant variety. For sufficiently small rank, which nevertheless improves upon the known bounds for specific identifiability, some local equations of this variety are known, allowing us to verify this property. As a particular example of our approach, we prove the identifiability of a specific $5\times 5\times 5$ tensor of rank $7$, which cannot be handled by the conditions recently provided in [I.~Domanov and L.~De Lathauwer, \textit{On the Uniqueness of the Canonical Polyadic Decomposition of third-order tensors---Part II: Uniqueness of the overall decomposition}, SIAM J.~Matrix Anal.~Appl.~34(3), 2013]. Finally, we also present a surprising new class of weakly-defective Segre varieties that nevertheless turns out to admit a generically unique decomposition.
\end{abstract}

\maketitle




\pagestyle{myheadings}
\thispagestyle{plain}
\markboth{CHIANTINI, OTTAVIANI, VANNIEUWENHOVEN}{GENERIC AND SPECIFIC IDENTIFIABILITY OF TENSORS}

\section{Introduction}
A tensor, which can be represented by a multidimensional array $\tensor{A} \in \C^{n_1 \times{} n_2 \times{} \cdots \times{} n_d}$ in fixed bases, where we assume without loss of generality that $n_1 \ge n_2 \ge \ldots \ge n_d$, is said to admit a \emph{rank-$r$ decomposition} whenever
\begin{align} \label{eqn:cpd}
 \tensor{A} = \sum_{i=1}^{r} \sten{a}{i}{1} \ktimes \sten{a}{i}{2} \ktimes \cdots \ktimes \sten{a}{i}{d},
\quad\text{with } \sten{a}{i}{\ell} \in \C^{n_\ell},\; \ell = 1, \ldots, d,
\end{align}
and where $\ktimes$ denotes the tensor product. In the above, $r$ is assumed to be minimal in the sense that no other decomposition of the above form with fewer terms exists: we say that the rank of $\tensor{A}$ is $r$. This general decomposition was introduced by Hitchcock \cite{Hitchcock1927aCOV,Hitchcock1927COV} and was later rediscovered several times, notably by Caroll and Chang \cite{Caroll1970COV}, who called it Candecomp, and by Harschman \cite{Harshman1970COV}, who called it Parafac. For this reason, the decomposition is also often called \emph{the CP decomposition}.

Essential uniqueness, or \emph{identifiability}, of the decomposition in \refeqn{eqn:cpd}, up to trivial indeterminacies, is one of its key properties in practice. According to Smilde, Bro, and Geladi \cite{Smilde2004COV}, the rank decomposition is nowadays widely used in chemistry, where it finds application in recovering the emission-excitation spectra of chemical components in a multicomponent fluorescent mixture. This idea was introduced in 1981 by Appellof and Davidson \cite{Appellof1981COV} who stated that ``the advantage of having a three-dimensional data matrix [relative to analyzing only two-dimensional excitation-emission matrices] is that if the factorization is found, it is unique.'' In a different context, according to Allman, Matias, and Rhodes \cite{Allman2009COV}, the identifiability of statistical models of type \refeqn{eqn:cpd} ``is a prerequisite of statistical parameter inference.'' 

Notwithstanding very substantial interest in identifiability \cite{Harshman1970COV,Kruskal1977COV,Strassen1983COV,ChCi2001COV,tBS2002COV,JS2004COV,EHN2005COV,ChCi2006COV,SS2007COV,Allman2009COV,Rhodes2010COV,CO2012COV,BCh2013COV,BCO2013COV,Domanov2013aCOV,Domanov2013bCOV,Bhaskara2013COV}, its theoretical foundations
are still not completely understood. A well-known condition for \emph{specific identifiability}---given a rank-$r$ decomposition, determine whether it is unique---was introduced by
Kruskal in \cite{Kruskal1977COV}. Letting $A^j = \begin{bmatrix} \sten{a}{l}{j}
\end{bmatrix}_{l=1}^{r}$, $j=1,2,3$, Kruskal's condition states that if 
\begin{equation}\label{eq:kruskalrange}
r \le \tfrac{1}{2} \left( k_{A^1} + k_{A^2} + k_{A^3} - 2 \right),
\end{equation}
where $k_{A^j}$ is the maximum number such that every set of $k_{A^j}$ columns of $A^j$
is linearly independent, then the decomposition given in \refeqn{eqn:cpd} is unique.
In addition to the question of specific identifiability, the condition also yields results about 
\emph{generic identifiability}---determine whether all rank-$r$ decompositions not in some 
set of measure zero are unique. From Kruskal's condition it follows that a generic 
rank-$r$ decomposition is unique if 
\[
 r \le \tfrac{1}{2} \left(\min(n_1,r) + \min(n_2,r) + \min(n_3,r) - 2\right).
\]
In the $n\times n\times n$ case, the above condition reduces to 
\(
 r \le \tfrac{1}{2} \left(3n - 2\right).
\)
It has been known since the work of Strassen \cite{Strassen1983COV} that Kruskal's condition, as well as the recent conditions by Domanov and De Lathauwer \cite[Table 6.2]{Domanov2013bCOV}, are quite weak for addressing \emph{the problem of generic identifiability}, at least for such cubic tensors. Strassen proved in \cite[Corollary 3.7]{Strassen1983COV} 
that a generic rank-$r$ decomposition in $\C^n \times \C^n \times \C^n$, $n$ odd, is unique 
whenever
\[
 r \le \left\lfloor\frac{n^3}{3n - 2} \right\rfloor -n,
\]
which is asymptotically better than Kruskal's condition by a factor $n$. This result
was recently extended to any $n$ in \cite[Corollary 6.2]{BCO2013COV}. 

We will investigate the identifiability of rank decompositions using techniques from algebraic geometry in this paper. Its language and terminology will be used, while attempting to maintain an exposition that requires no specialist knowledge. Before proceeding, some basic terminology is established. Recall from \cite{Landsberg2012COV} 
that a point $p_i \in \Var{S}$ on the Segre variety $\Var{S} = \Pj\C^{n_1}
\ttimes 
\Pj\C^{n_2} \ttimes \cdots \ttimes \Pj\C^{n_d}$ embedded in $\Pj\C^{n_1 n_2
\cdots 
n_d}$ can be parametrized by the tensors of rank $1$: we shall write $p_i = \sten{a}{i}{1} \ktimes \sten{a}{i}{2} \ktimes \cdots \ktimes \sten{a}{i}{d}$, with a slight abuse of notation, where $p_i$ is literally a representative
of the point, up to scalar multiples.\footnote{We also refer to \cite[Section 4.2]{Landsberg2012COV} for basic definitions in projective algebraic geometry.} 
A rank-$r$ decomposition is a linear combination of $r$
points $p_i \in \Var{S}$, where the number of summands $r$ is minimal. Geometrically, every rank-$r$ decomposition conforms to a point $p \in \Sec{r}{\Var{S}}$ on the 
\emph{$r$-secant variety} $\Sec{r}{\Var{S}}$ of the Segre variety
$\Var{S}$, which is defined as the closure in the Zariski topology of the set of linear combinations
of $r$ points on $\Var{S}$. Note that not every $p \in \Sec{r}{\Var{S}}$ has rank $r$, a situation arising from taking the closure in the Zariski or Euclidean topology, and which may lead to the \emph{ill-posedness} of the standard approximation problem associated with \refeqn{eqn:cpd}; see de Silva and Lim \cite{Silva2008COV} for more details in this regard. A Segre variety $\Var{S}$ is said to be \emph{generically $r$-identifiable} if a general element
of $\Sec{r}{\Var{S}}$ admits a unique representation as a linear combination of
points in $\Var{S}$; i.e., the representation in \refeqn{eqn:cpd} is unique up to the trivial scaling indeterminacies that arise when considering the rank decomposition in an affine setting. In other words, if $\Var{S}$ is generically $r$-identifiable, then there exists a set $M$ of Zariski, and, thus, Euclidean, measure zero, such that all elements of $\sigma_r(\Var{S}) \setminus M$ are $r$-identifiable. In particular, if we sample a ``random'' element on $\sigma_r(\Var{S})$, imposing any reasonable continuous probability distribution, then this element will be identifiable. Furthermore, and conversely, if $\Var{S}$ is generically $r$-identifiable and we have a nonidentifiable element $p \in \sigma_r(\Var{S})$, then there will exist, for every $\epsilon > 0$, points $p' \in \sigma_r(\Var{S})$ with $\|p-p'\| \le \epsilon$ and $p'$ $r$-identifiable, where the norm is the Euclidean norm. \emph{Nonidentifiable points are, thus, in a sense, nonstable points of a generically $r$-identifiable Segre variety $\Var{S}$}; a general infinitismal perturbation, on $\sigma_r(\Var{S})$, will make them $r$-identifiable.

In this paper, a new sufficient condition for generic 
identifiability is developed based on the geometrical concept of tangential 
weak defectivity, extending \cite{CO2012COV,BCO2013COV}. As the condition is more
involved to verify, an algorithm, based on familiar linear and multilinear operations, for 
testing the proposed condition is described in some detail. 
As generic $(r-1)$-identifiability is implied by generic
$r$-identifiability \cite{CO2012COV}, 
the application of this algorithm for the problem of generic identifiability will be limited to the largest $r$ possible, 
which is one less than the expected generic rank:
\[
\underline{r} = \left\lceil \frac{\Pi_{i=1}^{d} n_i}{1 + \Sigma_{i=1}^{d} (n_i -
1)} 
\right\rceil - 1;
\]
however, if the fraction is integer, then $\underline{r}$ should be one larger
than stated. In this case, one says that the Segre variety $\Var{S}$ has a perfect shape
\cite{Strassen1983COV, Lickteig1985COV}. Unfortunately, the proposed algorithm is not designed to handle the
$(\underline{r} + 1)$-secant in the case of perfect shapes. Our investigation will therefore be limited 
to $\underline{r}$ for all Segre varieties. We will say that tensors of rank $r\le \underline{r}$ are of \emph{subgeneric rank}. Remark that generic $r$-identifiability does not hold for $r$ strictly larger than $\underline{r}$, respectively $\underline{r}+1$ for perfect shapes, as is well known \cite[Proposition 3.3.1.2]{Landsberg2012COV}.

In \cite{BCO2013COV}, a list of all known cases where generic identifiability fails 
is presented. Using the proposed algorithm, we verified generic
$r$-identifiability for a large number of complex tensor spaces, providing additional evidence that
the list from \cite{BCO2013COV} is complete for the varieties tested. 
The main result we prove is:
\begin{theorem}\label{main1}
A generic tensor $\tensor{A} \in \C^{n_1 \times{} n_2 \times{} \cdots \times{}
n_d}$ 
of subgeneric rank $r\le \underline{r}$ is $r$-identifiable if
$\prod_{i=1}^dn_i\le 15000$ 
unless we have one of the following:
{\small
$$\begin{array}{ccl}
\toprule
(n_1,\ldots,n_d) & r & \mbox{type} \\
\midrule
(4,4,3) & 5 &\mbox{defective \cite{AOP2009COV}}\\
(4,4,4) & 6 &\mbox{sporadic \cite{CO2012COV}} \\ 
(6,6,3) & 8 &\mbox{sporadic \cite{CMOCOV}} \\
(n,n,2,2) & 2n-1 &\mbox{defective \cite{AOP2009COV}} \\
(2,2,2,2,2) & 5 &\mbox{sporadic \cite{BCh2013COV}} \\
n_1 > \prod_{i=2}^dn_i-\sum_{i=2}^d(n_i-1) & r\ge 
\prod_{i=2}^dn_i-\sum_{i=2}^d(n_i-1) & \mbox{unbalanced  \cite{BCO2013COV}} \\
\bottomrule
\end{array}$$
}
\end{theorem}

The theorem was stated with $100$ instead of $15000$ in \cite{BCO2013COV}.\footnote{The defective varieties $\Pj\C^{n} \times \Pj\C^n \times \Pj\C^3$, $n$ odd, appear to be missing relative to \cite{BCO2013COV}, however, that is because they are defective only in the $(\underline{r}+1)$-secant.} With the exception of the perfect shapes, \emph{these results are optimal} in the sense that generic $(\underline{r}+1)$-identifiability does not hold.\footnote{As a corollary, this also proves nondefectivity of the $\underline{r}$-secant variety of these Segre varieties, providing further evidence for the Abo--Ottaviani--Peterson conjecture \cite{AOP2009COV}, which already received a strong numerical confirmation in \cite{Vannieuwenhoven2012cCOV}.}

The algorithm presented in this paper allows us to treat a considerably larger number of cases, yielding
results we believe to be of practical relevance, because of an additional result that is implied by \refthm{preservesuniqueness} in \refsec{spec}:

\begin{corollary}
A generic tensor $\tensor{A} \in \C^{n_1 \times \cdots \times n_d}$ of multilinear rank $(r_1, \ldots, r_d)$ and of subgeneric rank $r \le \underline{r}$ in $\C^{r_1 \times \cdots \times r_d}$, i.e., 
\[
\underline{r} = \left\lceil \frac{ \prod_{i=1}^d r_i}{1 + \sum_{i=1}^d ( r_i - 1)} \right\rceil - 1,
\]
is $r$-identifiable if $\prod_{i=1}^d r_i \le 15000$.
\end{corollary}

In addition to generic identifiability, we also investigate whether the algorithm can be extended to handle the problem of specific identifiability. We will show that if a specific rank-$r$ decomposition, considered as a point on the $r$-secant variety of a Segre variety, is nonsingular, then the algorithm for generic identifiability may be applied. Unfortunately, little is known about the singularities of these varieties; nonetheless, local equations for secant varieties of low order can be obtained, allowing us to propose a test for nonsingularity of a given rank-$r$ decomposition. This technique allows to handle specific tensors that cannot be covered by the criterions of Kruskal and Domanov--De Lathauwer. In particular, we propose a specific example, in \refsec{sec_specific_identifiability}, of a $5\times 5\times 5$ tensor of rank $7$ that is proved to be identifiable.

{The remainder of the paper is structured as follows. In \refsec{sectionalgtesting}, a sufficient condition for generic $r$-identifiability is proposed, and a new class of identifiable but weakly-defective secant varieties is presented. Section \ref{thealgorithm} investigates an algorithm based on the proposed sufficient condition; \refthm{main1} is proved. A sufficient condition for specific $r$-identifiability is then proposed in \refsec{spec}. This condition is used in \refsec{sec_specific_identifiability} in combination with local equations for the $r$-secant variety to prove identifiability of a specific example beyond the criterions of Kruskal and Domanov--De Lathauwer. Finally, \refsec{sec_conclusions} presents our conclusions and open questions.}

\paragraph{Notational conventions}
We denote by $\mathrm{T}_{p} X$ the tangent space to an algebraic variety $X \subset \Pj\C^N$ in $p \in X$. We let
\[
 \Var{S} = \Pj\C^{n_1} \times \Pj\C^{n_2} \times \cdots \times \Pj\C^{n_d},
\quad n_1 \ge n_2 \ge \cdots \ge n_d,
\]
be the Segre variety under study, and define furthermore the constants 
\[
 \Pi = \prod_{i=1}^d n_i,
\quad
\Sigma = \sum_{i=1}^d (n_i - 1)
\quad\mbox{and}\quad
\underline{r} = \left\lceil \frac{\Pi}{1+\Sigma} \right\rceil - 1.
\]
Note that $\Var{S}$ has dimension $\Sigma$ and is naturally embedded in $\Pj\C^{\Pi}$. The $r$-secant variety of $\Var{S}$ is formally given by
\[
 \Sec{r}{\Var{S}} = \overline{ \bigcup_{p_1, \ldots, p_r \in \Var{S}} \langle p_1, p_2, \ldots, p_r \rangle } \subset \Pj\C^\Pi,
\]
where the line denotes the Zariski closure. The linear span of the spaces $L_i \subset V$, $i=1,\ldots,k$, is denoted by $\langle L_1, \ldots, L_k \rangle \subset V$.

\section{A sufficient condition for generic identifiability}\label{sectionalgtesting}
A symbolic algorithm implemented in Macaulay2 for verifying whether a generic rank-$r$ tensor is identifiable was sketched in \cite{BCO2013COV}. In essence, it augments the well known algorithm based on Terracini's lemma for verifying nondefectivity of the $r$-secant variety of a Segre variety, see, e.g., \cite{Abo2010COV,Vannieuwenhoven2012cCOV}, with an additional step verifying, essentially, that no other points on the Segre variety have their tangent space contained within the tangent space spanned by $r$ generic points on the Segre variety. In this section, we expound upon the correctness of the algorithm in \cite[Section 9]{BCO2013COV}, and present a sufficient condition for generic $r$-identifiability based entirely on basic linear algebra. 

The starting point of our investigation is Terracini's characterization of the tangent space at a general point on the $r$-secant variety of any variety \cite{Terracini1911COV,Zak1993COV}. We recall the result here, for we will need to refer often to the statement. It reads: 
\begin{lemma}[Terracini's lemma \cite{Terracini1911COV}] Let $\Var{S} \subset \Pj\C^\Pi$ be a Segre variety, let $p_1$, $p_2$, $\ldots$, $p_r \in \Var{S}$ be general points, and let $p \in \Sec{r}{\Var{S}}$ be general in $\langle p_1, p_2, \cdots, p_r \rangle$. Then, 
\[
 \mathrm{T}_p \Sec{r}{\Var{S}} = \langle \mathrm{T}_{p_1} \Var{S},
\mathrm{T}_{p_2} \Var{S}, 
 \ldots, \mathrm{T}_{p_r} \Var{S} \rangle;
\]
that is, the tangent space to the $r$-secant variety in $p$ is given by the linear span of the tangent spaces to the Segre variety in each of the $r$ points.
\end{lemma}

By definition, a generic rank-$r$ tensor with $r \le \underline{r}$ in $\Pj\C^{n_1 \times \cdots \times n_d}$ admits a unique representation as a sum of rank-$1$ tensors if and only if the \emph{$r$-secant order} of the Segre variety $\Var{S} = \Pj\C^{n_1} \times \cdots \times \Pj\C^{n_d}$ is one \cite{ChCi2006COV}. This concept is strongly related to \emph{$r$-weak defectivity} \cite{ChCi2001COV}; a variety $\Var{A}$ is said to be $r$-weakly defective if a general hyperplane containing the tangent space at $r$ general points of $\Var{A}$ is also tangent to the variety in another point distinct from these $r$ points. It was proved in \cite{ChCi2006COV} that a variety that is not $r$-weakly defective has $r$-secant order one. In Proposition 2.4 in \cite{CO2012COV}, the notion of \emph{not $r$-tangential weak defectivity}, which entails not $r$-weak defectivity, was introduced. This is the key geometrical property that the algorithm from \cite{BCO2013COV} exploits. The proposition from \cite{CO2012COV} states:
\begin{proposition}[Chiantini and Ottaviani \cite{CO2012COV}] \label{prop:CO}
 Let $p_1, p_2, \ldots, p_r \in \Var{S}$ be $r \le \underline{r}$ particular points of a Segre variety $\Var{S} \subset \Pj\C^{\Pi}$ whose $r$-secant variety is nondefective and $p \in \Var{S}$ any point. 
Let $\mathrm{H} = \langle \Tang{p_1}{\Var{S}}, \Tang{p_2}{\Var{S}}, \ldots, \Tang{p_r}{\Var{S}} \rangle$.
If $\{ p \in \Var{S} \;|\; \Tang{p}{\Var{S}}\subseteq H\}$ consists only of the simple points $\{p_1, p_2, \ldots, p_r\}$, then the Segre variety $\Var{S}$ is $r$-identifiable.
\end{proposition}

A Segre variety $\Var{S}$ is said to be \emph{not $r$-tangentially weakly defective} whenever the condition in the above proposition holds. Similar in spirit to Terracini's lemma, this proposition reduces the problem of investigating generic identifiability of an algebraic variety, which is a global property, to a local computation. 
We can reduce this check further to an infinitesimal computation that can be performed at a given point $p_1$. To this end, we recall the definition 
of \emph{$r$-tangential contact locus} $\Var{C}_r$ from
\cite{BCO2013COV}:
\begin{align} \label{eqn:contact-var-def}
 \Var{C}_r = \left\{ p \in \Var{S} \;|\; \Tang{p}{\Var{S}} \subset \Plane{H} =
\langle \Tang{p_1}{\Var{S}}, \ldots, \Tang{p_r}{\Var{S}} \rangle \right\}
\subset \Var{S} \subset \Pj\C^\Pi.
\end{align}
When no ambiguity arises, we denote $\Var{C}=\Var{C}_r$. 
The next proposition is the prime ingredient of the newly proposed sufficient condition.

\begin{proposition}\label{infinitepoints}
Let $\Var{S}$ be a nondefective Segre variety, let $r \le \underline{r}$ and assume that it is not $r$-identifiable. Then, for $r$ general points $p_1, p_2, \ldots, p_r \in \Var{S}$, the \emph{$r$-tangential contact locus} 
$\Var{C}_r$ contains a curve, passing through $p_1, p_2, \ldots, p_r $. 
\end{proposition}
\begin{proof} 
If $\Var{S}$ is not $r$-identifiable, then we have in affine notation
\[
p = \sum_{i=1}^r a_i p_i = \sum_{i=1}^r b_i q_i
\]
with $a_i, b_i\in\C$, $q_i \in \Var{S}$. At least one of the $q_i \notin \{p_1,\ldots,p_r\}$ for if $p$ would have two different expressions as a linear combination of the $p_i$, it would follow that the $p_i$ do not form a linearly independent set, contradicting the generality of the $p_i$. In fact, it would imply that $p$ is an element of the $(r-1)$-secant variety. By the generality of the points, Terracini's lemma applies, so that $\Tang{p}{\Sec{r}{\Var{S}}} = \Plane{H}$. Letting $b_i(t) \ne 0$ be a curve with a parameter $t$ in a neighborhood of $0$, in such a way that $b_i(0)=b_i$, the resulting tensor $p(t)= \sum_{i=1}^r b_i(t) q_i$ has a tangent space $\Tang{p(t)}{\Sec{r}{\Var{S}}}$ which is constant with respect to $t$ by Terracini's lemma. We can then choose $b_i(t)$ in such a way that $p(t) \notin \langle p_1,\ldots, p_r \rangle$, because otherwise the (generalized) Trisecant Lemma, see, e.g., Theorem 2.6 in \cite{ChCi2001COV}, would be contradicted as we would have that $\langle p_1,\ldots, p_r \rangle = \langle q_1,\ldots, q_r \rangle$ for general points. By the assumption of not $r$-identifiability and nondefectivity of $\Var{S}$, we may thus write
\[
p(t) = \sum_{i=1}^r a_i(t) p_i(t) \quad\mbox{with } a_i(0) = a_i \mbox{ and }
p_i(0) = p_i,
\]
where, by the previous argument, not all $p_i(t)$ can be constant. Then, we have infinitely many $p_1(t)$ such that $\Tang{p_1(t)}{\Var{S}} \subset \Tang{p(t)}{\Sec{r} {\Var{S}}} = \Plane{H}$. By monodromy we get infinitely many $p_i(t)$ such that $p_i(0) = p_i$ and $\Tang{p_i(t)}{\Var{S}} \subset H$ for any $i$. This concludes the proof.
\qquad\end{proof}

Note that we do not claim irreducibility of the tangential contact locus: it may have many components. In this case, however, since we can interchange by monodromy any couple of points $p_i,p_j$, it follows that the tangential contact locus has one component through every point $p_i$, as explained in Proposition 2.2 of \cite{ChCi2006COV}.

The algorithm in \cite{BCO2013COV} explicitly constructs Cartesian equations for the \emph{$r$-tangential contact locus} $\Var{C}_r$, as in (\ref{eqn:contact-var-def}), for the nondefective Segre variety $\Var{S}$.
The dimension of $\Var{C}_r$ equals the dimension of the tangent space at a general point, and by the generality of the points $p_1, \ldots, p_r$, we can can compute it, for the sake of simplicity, at $p_1$. From \refprop{infinitepoints} it follows that $\Var{S}$ is $r$-identifiable if $\Var{C}_r$ is zero-dimensional at $p_1$.

The gist of the algorithm in \cite{BCO2013COV} concerns the construction of the equations for $\Var{C}$. Consider the Segre embedding:
\begin{align*}
 s: \C^{n_1} \times \C^{n_2} \times \cdots \times \C^{n_d} &\to
\C^{\Pi} \\
(\sten{a}{ }{1}, \sten{a}{ }{2}, \ldots, \sten{a}{ }{d}) &\mapsto  \sten{a}{ }{1} \ktimes \sten{a}{ }{2} \ktimes \cdots \ktimes
\sten{a}{ }{d}
\end{align*}
whose image is the affine cone over the Segre variety.
For notational convenience, we let $m(\cdot)$ denote the bijection between linear and multilinear indices, such that $x_{m(i_1, \ldots, i_d)} = a_{i_1}^1 \cdots a_{i_d}^d$ whenever $\mathbf{x} = \sten{a}{ }{1} \ktimes \cdots \ktimes \sten{a}{ }{d}$. We say that the source of $s$ provides a parameterization of the points on the affine cone over the Segre variety $\Var{S}$. First, a particular $\Plane{H}$ is constructed by choosing $r$ particular points $p_1, \ldots, p_r \in \Var{S}$ and considering the span of the tangent spaces in these points. We may assume without loss of generality that $p_1 = \vect{e}_{1}^1 \ktimes \cdots \ktimes \vect{e}_{1}^d$, where $\vect{e}_{1}^{i}$ is the first standard basis vector of the corresponding vector space $\C^{n_i}$. Suppose that the $\ell$ independent Cartesian equations of $\Plane{H} \subset \Pj\C^\Pi$ are:
\begin{align*}
q_l (x_1, x_2, \ldots, x_\Pi) = \sum_{i=1}^\Pi k_{i,l} x_i = \sum_{i_1=1}^{n_1} 
\sum_{i_2=1}^{n_2} \cdots \sum_{i_d=1}^{n_d} k_{m(i_1, i_2, \ldots, i_d),l} 
x_{m(i_1,i_2,\ldots,i_d)} = 0,
\end{align*}
for $l = 1,2,\ldots, \ell$, wherein the coefficients $k_{i,l}$ are constants, because the choice of the particular points $p_1, \ldots, p_r$ is fixed. Note that $\Plane{H}$ is thus an $(\Pi-\ell)$-dimensional linear subspace of $\Pj\C^\Pi$. Then, the intersection of a general point $p = \sten{a}{ }{1} \ktimes \sten{a}{ }{2} \ktimes \cdots \ktimes \sten{a}{ }{d} \in \Var{S}$, assuming without loss of generality that $a_{1}^k = 1$ for $k = 1, \ldots, d$, with $\mathrm{H}$ can be parameterized by simple substitution:
\begin{align}\label{eqn:polynomial-eqns}
q_l ( \sten{a}{ }{1}, \sten{a}{ }{2}, \ldots, \sten{a}{ }{d} ) = 
\sum_{i_1=1}^{n_1} \sum_{i_2=1}^{n_2} \cdots \sum_{i_d=1}^{n_d} k_{m(i_1, i_2,
\ldots, 
i_d),l} a_{i_1}^1 a_{i_2}^2 \cdots a_{i_d}^d = 0,
\end{align}
for $l = 1,2,\ldots,\ell$. Interestingly, to impose that $\Tang{p}{\Var{S}} \subset \mathrm{H}$, it suffices, due to linearity, that each of the basis vectors in the tangent space $\Tang{p} {\Var{S}}$ satisfies the above Cartesian equations. 
An explicit description of the tangent space is readily obtained by taking partial derivatives with respect to the parameters $\vect{a} = \sten{a}{ }{1}, \ldots, \sten{a}{ }{d}$ of the equations for the Segre variety: $x_{m(i_1, \ldots, i_d)} = a_{i_1}^1 a_{i_2}^2 \cdots a_{i_d}^d$. By the linearity of the Cartesian equations, it follows that we may simply take partial derivatives of \refeqn{eqn:polynomial-eqns} with respect to the parameters $\vect{a}$ to obtain the equations for $\Var{C}$. We will concisely write: 
\begin{align} \label{eqn:contact-var-eqns}
\{ \tfrac{\partial}{\partial \vect{a}} q_l (\vect{a}) = 0 \}_{l=1}^\ell,
\end{align}
where $\frac{\partial}{\partial \vect{a}} q_l (\vect{a})$ represents the system of equations obtained by partial derivation of \refeqn{eqn:polynomial-eqns} to each of the parameters $\vect{a}$. Given these equations of $\Var{C}$, the dimension of this algebraic variety is obtained, by definition, as the dimension of the (linear) tangent space in a general point of $\Var{C}$, which is again obtained by taking partial derivatives with respect to the parameterization. We note that this corresponds to computing the Jacobian of the above equations, or, equivalently, the ``stacked Hessian'' of the multivariate homogeneous polynomial \refeqn{eqn:polynomial-eqns} evaluated in a general point of $\Var{C}$. We choose to evaluate it in the point $p = p_1 \in \Var{C}$. The stacked Hessian $H = \left[\begin{smallmatrix}H^1 & H^2 & \cdots & H^\ell \end{smallmatrix}\right]$ is a block matrix wherein every block corresponds to the double partial derivation of $q_l$ with respect to the parameters; i.e., $H^l$ is the Hessian of $q_l$ evaluated in $p_1$. These Hessians admit an additional block structure:
\begin{align} \label{eqn:one-hessian}
 H^l = 
\begin{bmatrix}
H^l_{11} & H^l_{12} & \cdots & H^l_{1d} \\[3pt]
H^l_{21} & H^l_{22} & \cdots & H^l_{2d} \\[3pt]
\vdots   & \vdots   & \ddots & \vdots   \\[3pt]
H^l_{d1} & H^l_{d2} & \cdots & H^l_{dd} \\[3pt]
\end{bmatrix}
\quad\mbox{where}\quad
(H^l_{IJ})_{ij} = \frac{\partial^2 q_l}{\partial a_{1+i}^{I} \partial
a_{1+j}^{J}} \;
\Biggr|_{p=p_1},
\end{align}
for $1 \le I, J \le d$ with $i = 1, \ldots, n_I-1$ and $j = 1, \ldots, n_J-1$. Note that $H^l \in \mathbb{C}^{\Sigma \times \Sigma}$ because $a_{1}^{k} = 1$; thus, we need not derive to it. From \refeqn{eqn:polynomial-eqns} it is also clear that deriving twice in mode $I$, i.e., to $a_{i}^{I}$ and $a_j^{I}$, is zero, because none of the terms has two variables from the same mode. This explains why the block diagonal of $H^l$ is, in fact, zero: $H^l_{II} = 0$. It is straightforward to verify that all nonconstant terms in \refeqn{eqn:polynomial-eqns} after the double partial derivation are zero due to the special choice of $p_1$. As a result, the off-diagonal block matrices $H_{IJ}^l$, $I \ne J$, are given explicitly by \begin{align} \label{eqn:hessian-offdiagonal-block} (H_{IJ}^l)_{ij} = k_{m(1,\ldots,1,i+1,1,\ldots,1,j+1,1,\ldots,1),l} \end{align} where $i$ is at position $I$ and $j$ at position $J$ in the multi-index. 

The rank of $H$ reveals the local codimension of $\Var{C}$; we recall that $\Var{C}$ is specified by the Cartesian equations, so that its dimension is is the dimension of $\Var{S}$ minus the number of independent additionally imposed conditions which are given by \refeqn{eqn:contact-var-eqns}. If $H$ is of maximal rank, we can be sure that $p = p_1$ is a general point and thus that the local dimension equals the global dimension.\footnote{This is easy to understand from the fact that the elements of $H$ are multivariate polynomials in the variables $\vect{a}$. Consider the set of $\Sigma \times \Sigma$ minors of $H$, then the determinant is also a multivariate polynomial in the parameters $\vect{a}$ and at least one of them is nonzero in $p_1$ because $H$ is of maximal rank. The existence of an $\epsilon$-neighborhood around $\vect{a}$ where this property is maintained follows immediately.} On the other hand, if the rank of $H$ is not maximal, the algorithm is unable to conclude that the Segre variety $\Var{S}$ is $r$-identifiable. This problem may have arisen from an unfortunate choice of initial points $p_1, \ldots, p_r$, so it may be advised to rerun the algorithm several times. If in none of these runs $H$ has maximum rank, this may be taken as an indication that the Segre variety $\Var{S}$ is $r$-tangentially weakly defective, in which case it may or may not be $r$-identifiable. Conversely, if $H$ has the maximum dimension, $\Var{C}$ is zero-dimensional, and we may conclude that $\Var{S}$ is $r$-identifiable by \refprop{infinitepoints}. 
For future reference, we state this as 
\begin{proposition}\label{prop_reformulation}
Let $\Var{S} \subset \Pj\C^\Pi$ be a Segre variety of dimension $\Sigma$. Let $H = [ \begin{smallmatrix} H^1 & \cdots & H^\ell \end{smallmatrix} ]$ be the stacked Hessian with $H^l$ as in \refeqn{eqn:one-hessian}. Assume that the rank of $H$ is maximal, i.e., equal to $\Sigma$, then $\Var{S}$ is (generically) $r$-identifiable.
\end{proposition}

We left two items unspecified thus far: firstly, we did not mention how to assess that the $r$-secant variety is nondefective, which is required for \refprop{infinitepoints} to be applicable; and, second, the construction of the equations for the kernel was not detailed. We tackle both issues concurrently.  Recall that $\Sec{r}{\Var{S}}$ is nondefective if and only if its dimension is maximal. For verifying this property, Terracini's lemma is typically exploited to reduce the computation of the tangent space at a point on $\Sec{r}{\Var{S}}$ to computing the span of tangent spaces to $r$ general points on the Segre variety; see, e.g., \cite{Abo2010COV,Comon2009COV,Vannieuwenhoven2012cCOV}. Recall that the linear space $\mathrm{H}$ under consideration in \refprop{prop:CO} and \ref{infinitepoints} is exactly equal to $\Tang{p}{\Sec{r}{\Var{S}}}$. For computing the Cartesian equations of this space practically, we note that $\mathrm{H}$ corresponds to the column span of some matrix $T \in \C^{\Pi \times r(\Sigma+1)}$; hence, the coefficients of the Cartesian equations can be found as any set of basis vectors for the kernel of $T$, which can be obtained by applying Gaussian elimination to the extended system $\left[\begin{smallmatrix}T & I\end{smallmatrix}\right]$. If $K = \left[\begin{smallmatrix} \vect{k}_1 & \vect{k}_2 & \cdots & \vect{k}_\ell \end{smallmatrix}\right]$ is a basis for the null space thus obtained, then the columns of $K$ are the coefficients in \refeqn{eqn:polynomial-eqns}. The test for nondefectivity consists of verifying that $\Pi - r(\Sigma+1) = \ell$. If this equality is not satisfied, the algorithm cannot conclude that the Segre variety is $r$-identifiable: the lower rank may have been caused either by an unfortunate selection of the initial points $p_1, p_2, \ldots, p_r$, or by a defective $r$-secant. As $r$-defective secant varieties are not generically $r$-identifiable \cite{Landsberg2012COV}, the algorithm must stop here. Finally, the required matrix representation of $\Plane{H}$ is easily constructed. It is well-known, see, e.g., \cite{AOP2009COV,Landsberg2012COV,Vannieuwenhoven2012cCOV}, that the span of $\mathrm{T}_{p_i} \Var{S}$, $i = 1, \ldots, r$, is represented by the column span of
\begin{align}\label{eqn:tangent-1}
 T_i = \begin{bmatrix}
        T_i^1 & T_i^2 & \cdots & T_i^d
       \end{bmatrix}
\mbox{ with }
T_i^k = \sten{a}{i}{1} \ktimes \cdots \ktimes \sten{a}{i}{k-1} \ktimes I_{n_k} \ktimes \sten{a}{i}{k+1} \ktimes \cdots \ktimes \sten{a}{i}{d},
\end{align}
and $I_{n_k}$ represents an identity matrix of order $n_k$. By Terracini's lemma, the span of $\mathrm{H} = \mathrm{T}_p {\Sec{r}{\Var{S}}}$ coincides with the column span of 
\begin{align}\label{eqn:tangent-2}
 T' = \begin{bmatrix} T_1 & T_2 & \cdots & T_r \end{bmatrix},
\end{align}
provided that the $p_i$ are sufficiently general. $T'$ is overparameterized, but a simple permutation matrix $P$ can reduce $T = T' P$ to a $\Pi \times r(\Sigma+1)$ matrix with the same column span. Generically, it suffices to remove the last column from $T_i^k$, $i=1,\ldots,r$, $k=2,\ldots,d$; see \cite{Vannieuwenhoven2012cCOV} for more details.\footnote{It is not mandatory to work with a basis for representing the span of $T$, but we found that it simplified the programming and improves the efficiency of the code.}

\subsection{Unanticipated weakly defective varieties}\label{newclass}
Using the approach outlined above, we encountered several previously unknown $\underline{r}$-tangentially weakly defective, and thus $\underline{r}$-weakly defective, varieties. Remarkably, all of the discovered cases were of the same type, and we will show that \emph{generic $\underline{r}$-identifiability can still hold} if the sufficient condition of Theorem \ref{th:fakeareok} about the rank of the stacked Hessian is satisfied. These examples were not detected in \cite{BCO2013COV} because only Segre varieties embedded in $\Pj\C^{\Pi}$ with $\Pi \le 100$ were investigated there, while the smallest instance of this new class occurs in the space $\Pj\C^{144}$, namely for the Segre variety $\Pj\C^{8} \times \Pj\C^{3} \times \Pj\C^{3} \times \Pj\C^{2}$.

The key observation that characterizes all of the unanticipated cases is  
\begin{lemma} \label{hessianconditions}
Let $\Var{S}$ be a nondefective Segre variety, $\ell = \Pi - \underline{r}(\Sigma+1)$ and
\[
 n_1 - 1> \ell \sum_{i=2}^d (n_i - 1),
\]
then the stacked Hessian $H = \begin{bmatrix} H^1 & H^2 & \cdots & H^\ell \end{bmatrix}$, with $H^l$ as in \refeqn{eqn:one-hessian}, is not of full rank. Instead, its rank is bounded from above by $(\ell+1)\sum_{i=2}^d (n_i - 1)$.
\end{lemma}
\begin{proof}
We can rearrange the columns of $H$ by applying an $\ell\Sigma \times \ell\Sigma$ permutation matrix $P'$ on the right, so that we find
\begin{align*}
\setlength{\extrarowheight}{3pt}
 HP' &= 
\left[\begin{array}{ccccccc|ccc}
 H^1_{12} & \cdots & H^1_{1d} & \cdots & H^\ell_{12} & \cdots & H^\ell_{1d} &
H^1_{11} & \cdots & H^\ell_{11}  \\[1pt]
\hline &&&&&&&\\[-8pt]
 H^1_{22} & \cdots & H^1_{2d} & \cdots & H^\ell_{22} & \cdots & H^\ell_{2d} &
H^1_{21} & \cdots & H^\ell_{21} \\
 \vdots   & \ddots & \vdots   & \cdots & \vdots      & \ddots & \vdots      &
\vdots   & \ddots & \vdots \\
 H^1_{d2} & \cdots & H^1_{dd} & \cdots & H^\ell_{d2} & \cdots & H^\ell_{dd} &
H^1_{d1} & \cdots & H^\ell_{d1} \\
\end{array}\right]\\
&= 
\left[\begin{array}{c|c}
H_{11}' & 0 \\[1pt]
\hline &\\[-8pt]
H_{21}' & H_{22}'
\end{array}\right],
\end{align*}
where $H$, and $HP'$, are $\Sigma \times \ell \Sigma$ matrices, $H_{11}'$ is $(n_1-1) \times \ell\sum_{i=2}^d (n_i-1)$, $H_{21}'$ is $\sum_{i=2}^d (n_i-1) \times \ell \sum_{i=2}^d (n_i-1)$, and $H_{22}'$ is $\sum_{i=2}^d (n_i-1) \times \ell (n_1 - 1)$. The rank of $\left[\begin{smallmatrix} H_{11}' & 0 \end{smallmatrix}\right]$ is the rank of $H_{11}'$, which is at most $\ell \sum_{i=2}^d (n_i - 1)$ because it is the smaller of the two dimensions provided that the condition in the lemma holds. The rank of $\left[\begin{smallmatrix}H_{21}' & H_{22}'\end{smallmatrix}\right]$ is bounded by $\sum_{i=2}^d (n_i - 1)$, because it is the smaller of the two dimensions. Combining above upper bounds for the rank concludes the proof.
\qquad\end{proof}

As an immediate corollary, we obtain\footnote{Note that the corollary exploits the observation that defective varieties are also generically $\underline{r}$-tangentially weakly defective.}

\begin{corollary} \label{cor:triangle-inequality}
A (possibly defective) Segre variety $\Var{S}$ satisfying the arithmetic conditions in \reflem{hessianconditions} is $\underline{r}$-tan\-gen\-tial\-ly weakly defective, and, hence, 
$\underline{r}$-weakly defective.
\end{corollary}

As the arithmetical condition on the size of $n_1$ in \reflem{hessianconditions} cannot be satisfied by replacing $\ell$ with $\ell+a(\Sigma+1)$ for some $a\ge 1$, this corollary can only show that some Segre varieties are $\underline{r}$-tangentially weakly defective: it never applies for $r < \underline{r}$.

We mentioned before that not $\underline{r}$-tangentially weakly defective is a sufficient condition for generic identifiability; now, it will be shown that it is not necessary. The following theorem namely states that Segre varieties satisfying the conditions of \reflem{hessianconditions} can still be generically $\underline{r}$-identifiable.

\begin{theorem}\label{th:fakeareok}
 Let $\Var{S}$, $\ell$ and $H$ be as in \reflem{hessianconditions}. Assume $\Var{S}$ is not $\underline{r}$-defective. If the rank of the stacked Hessian $H$ is precisely
\begin{align*}
 (\ell+1)\sum_{i=2}^d (n_i - 1),
\end{align*}
then $\Var{S}$ is $\underline{r}$-(tangentially) weakly defective but nevertheless still $\underline{r}$-identifiable.
\end{theorem}
\begin{proof}
Let $p_1,\ldots, p_r\in\Var{S}$ be general points. 
{By \cite[Theorem 3.3]{Ott2013COV}, the $1$-tangential contact locus, say at the point $p_1$,
contains
a linear space of dimension $(n_1-1)-\sum_{j=2}^d (n_j - 1)$ passing through $p_1$ and which is contained in the linear space $\Pj^{n_1-1}\subset \Var{S}$ through $p_1$. The same argument,
applied $\ell$ times, to $\ell$ independent hyperplanes defining
the span $\langle T_{p_1}\Var{S},\ldots, T_{p_r}\Var{S}\rangle$, shows that
the $r$-tangential contact locus contains the disjoint\footnote{Terracini's lemma is applicable because the points are general.} union of $r$ linear spaces $L_i$, for $i=1,\ldots, r$, of dimension $(n_1-1)-\ell \sum_{j=2}^d (n_j - 1)$, where $p_i\in L_i$ and $L_i$ is contained in the corresponding linear space $\Pj^{n_1-1}\subset\Var{S}$ passing
 through $p_i$.} The assumption on the rank shows that each of $L_i$ is a irreducible component of the $r$-tangential contact locus; more precisely, around each $p_i$, the $r$-tangentially contact locus 
locally coincides with $L_i$. If $\Var{S}$ were not $r$-identifiable, then, by the same argument in the proof of \refprop{infinitepoints},
we would get different decompositions $\sum_{i=1}^ra_iv_i$ with $v_i\in L_i$. But the spaces $L_i$ span a subspace of maximal dimension, because each $L_i\subset T_{p_i}\Var{S}$ and the subspaces $T_{p_i}\Var{S}$ span a subspace of maximal dimension.
Then, we would have uniqueness of decomposition, which is a contradiction. 
\qquad\end{proof}

\section{An algorithm verifying generic identifiability}\label{thealgorithm}
In the previous section, it was explained in detail how the sufficient condition in \refprop{infinitepoints} and \refthm{th:fakeareok} can be verified in practice, given a collection of $r$ general points on $\Var{S}$. Based on the above considerations, the algorithm we propose for verifying generic identifiability is summarized in \refalg{alg_generic_uniqueness}.
\begin{algorithm}
\caption{An algorithm for verifying generic uniqueness}
\label{alg_generic_uniqueness}
\begin{enumerate}
 \item[S1.] Choose $r$ random points $p_i \in \Var{S}_{\mathbb{Z}_q} = \mathbb{Z}_q^{n_1} \times \cdots \times \mathbb{Z}_q^{n_d}$.
 \item[S2.] Construct the tangent space matrix $T \in \Z_q^{\Pi \times r(\Sigma+1)}$ following \refeqn{eqn:tangent-1} and \refeqn{eqn:tangent-2}.
 \item[S3.] Construct the extended matrix $Y = \left[\begin{smallmatrix} T & I_{\Pi} \end{smallmatrix}\right]$, with $I_{\Pi}$ the $\Pi \times \Pi$ identity matrix. Perform row-wise Gaussian elimination (without column pivoting) to reduce $Y$ to row-echelon form.
 \item[S4.] Extract the null space matrix $K^T$ as the $l \times \Pi$ lower-right submatrix, where $l$ is maximal such that the $l \times r(\Sigma+1)$ lower-left submatrix is zero. 
 \item[S5.] If $l > \ell$, the Segre variety may be $r$-defective; the algorithm halts and claims it cannot prove $r$-identifiability for this choice of points. On the other hand, if $l = \ell$, as expected, then continue with the next step.
 \item[S6.] Construct the Hessian matrix $H^k$, $k = 1, \ldots, \ell$, following \refeqn{eqn:one-hessian} and \refeqn{eqn:hessian-offdiagonal-block}.
 \item[S7.] Compute the rank $r$ of the stacked Hessian $H$ by performing Gaussian elimination. We distinguish between two cases:
 \begin{itemize}
  \item[S7a.] Assume that the shape of $\Var{S}$ satisfies the condition in \reflem{hessianconditions}. If the rank $r$ satisfies \refthm{th:fakeareok}, then the algorithm has proved $r$-identifiability; otherwise, it claims that it cannot prove $r$-identifiability for these points.
  \item[S7b.] Assume that the shape of $\Var{S}$ does not satisfy the condition in \reflem{hessianconditions}. If the rank $r$ is maximal, i.e., $\Sigma$, then the algorithm has proved $r$-identifiability; otherwise, it claims that $r$-identifiability cannot be proved with these points.
  \end{itemize}
\end{enumerate}
\end{algorithm}

Note, in \refalg{alg_generic_uniqueness}, that we propose to verify generic $r$-identifiability of $\Var{S}$ by computations over the finite field\footnote{The correctness of this approach should be clear from the observation that all of the computed matrices over $\Z_q$, i.e., $T$, $K$, and $H$, are equivalent to the same matrices computed over $\C$ modulo $q$, combined with the fact that if any of these matrices are of full rank in $\Z_q$, then necessarily so in $\Z$ (and $\C$) as well. Note that the converse of the last statement is not true: if $H$ is not of full rank in $\Z_q$, then this should not be interpreted as evidence that the Segre variety is not identifiable.} $\Z_q$, with $q$ prime, for the following computational reason: with finite field computations the number of bits for representing one number remains constant throughout the execution of the algorithm; the number of bits required is the number of bits to represent $q-1$. This advantage does not hold for computations over $\Z$, $\mathbb{Q}$, or $\C$: the storage bit-complexity of basic Gaussian elimination is not constant but rather a function of the size of the matrix. In fact, to obtain a nonexponential storage bit-complexity\footnote{See \cite{Fang1997COV} for some specific examples.} some nontrivial modifications to the algorithm are necessary, see, e.g., \cite{Bareiss1968COV,Dixon1982COV}.

One may wonder why we do not consider an implementation with, e.g., double precision floating-point arithmetic, which leads to faster algorithms in practice. The problem with such an approach is the occurrence of roundoff errors, which necessitates a numerical analysis for investigating their propagation throughout the algorithm. In \cite{Vannieuwenhoven2012cCOV}, such an approach was pursued, leading to probabilistic statements about nondefectivity of secant varieties of Segre varieties; however, we believe that approach is more involved than the one proposed here.

\subsection{Experimental results}
The above algorithm was implemented in C++ using the Eigen matrix library \cite{EigenCOV}. The code for computing over the finite field $\Z_q$ with $q = 2^{7} - 1 = 127$ was also developed, as Eigen has no native support for this. This particular finite field with $q$ a Mersenne prime was selected because these primes have some favorable computational properties with respect to the modulus operations. A C++ code implementing \refalg{alg_generic_uniqueness} is included in the ancillary files accompanying this paper. The algorithm we provide along with the manuscript handles the setting in which the Hessian criterion is verified at every $p_1, \ldots, p_r$
(not only at $p_1$), so that not all optimizations discussed in \refsec{sectionalgtesting} apply. With this code one can also verify the example presented in \refsec{sec_specific_example}.

Recall that generic $\underline{r}$-identifiability implies generic $r$-identifiability for all $r < \underline{r}$, so that we may restrict ourselves to the case of $\underline{r}$. With the algorithm presented here, generic $\underline{r}$-identifiability was assessed for all spaces with $\Pi \le 15000$; the largest number of factors tested was $13$ for the variety $(\Pj\C^{2})^{13}$. These experiments extend \cite{BCO2013COV} by two orders of magnitude. In particular, all results pertaining to Segre varieties with at least $7$ factors are original, as well as most results for less factors. 

In all of the $75993$ tested spaces, the algorithm proved generic $\underline{r}$-identifiability;\footnote{In the count of tested spaces, we do not include all of the known exceptions that were presented in \refthm{main1}.} these results include the spaces $\Pj\C^n \times \Pj\C^n \times \Pj\C^3$ with $n$ odd, which are $\underline{r}$-identifiable, but not $(\underline{r}+1)$-identifiable because the variety is defective for that rank. In $973$ cases, \reflem{hessianconditions} applied, and, hence, step S7a in \refalg{alg_generic_uniqueness} proved generic identifiability; in all other cases, generic identifiability was proved through step S7b. For the spaces $\Pj\C^n \times \Pj\C^n \times \Pj\C^2$, we could not always prove $\underline{r}$-identifiability in a small number of attempts; therefore, $\underline{r}$-identifiability in these spaces was established by considering computations over the larger prime field $\Z_{8191}$.

For the sake of completeness, we present in \reftab{tab_condition_comparison} results analogous to Table 6.2 in \cite{Domanov2013bCOV}, comparing the maximum rank for which generic $r$-identifiability holds, according to the Domanov--De Lathauwer sufficient condition \cite{Domanov2013bCOV}, which improves Kruskal's condition \refeqn{eq:kruskalrange}, and according to the sufficient criterion presented in this paper. From the table one learns that the proposed sufficient condition considerably improves upon the best results from \cite{Domanov2013bCOV}; in fact, aside from the perfect shapes, our results are optimal.
We remark that in the first row of \reftab{tab_condition_comparison}, it is well known that generic $(\underline{r}+1)$-identifiability holds, as can be detected using Domanov--De Lathauwer's criterion from \cite{Domanov2013bCOV}, while our algorithm cannot provide an answer in this case because they are perfect shapes.

\begin{table} \footnotesize
\newcommand{\ull}[1]{\mathit{#1}}
\newcommand{\exc}[1]{\mathbf{#1}}
\caption{A comparison between the maximum rank for which generic $r$-identifiability can be proved for $\Var{S} = \Pj\C^{m} \times \Pj\C^{n} \times \Pj\C^{n}$ using the Domanov--De Lathauwer criterion \cite[Proposition 1.31 and Table 6.2]{Domanov2013bCOV} ($\diamondsuit$) and the sufficient condition verified by \refalg{alg_generic_uniqueness} ($\clubsuit$). A maximum rank displayed in a slanted font indicates that the value is optimal. In bold face the {possibly suboptimal maximum rank values, all of which correspond to perfect shapes,} are highlighted in the case of \refalg{alg_generic_uniqueness}.}
\label{tab_condition_comparison}
\begin{center}
\begin{tabular}{r c rr c rr c rr c rr c rr c rr}
\toprule
$m$ & \multicolumn{17}{c}{$n$} \\
\cmidrule{3-19}
 && \multicolumn{2}{c}{$4$} && \multicolumn{2}{c}{$5$} && \multicolumn{2}{c}{$6$} && \multicolumn{2}{c}{$7$} && \multicolumn{2}{c}{$8$} && \multicolumn{2}{c}{$9$} \\
\cmidrule{3-4} \cmidrule{6-7} \cmidrule{9-10} \cmidrule{12-13} \cmidrule{15-16} \cmidrule{18-19}
 && $\diamondsuit$ & $\clubsuit$ && $\diamondsuit$ & $\clubsuit$ && $\diamondsuit$ & $\clubsuit$ && $\diamondsuit$ & $\clubsuit$ && $\diamondsuit$ & $\clubsuit$ && $\diamondsuit$ & $\clubsuit$  \\
\midrule
2 && $\ull{4}$ & $\exc{3}$ 		&& $\ull{5}$ & $\exc{4}$   && $\ull{6}$ & $\exc{5}$   && $\ull{7}$ & $\exc{6}$   && $\ull{8}$ & $\exc{7}$   && $\ull{9}$ & $\exc{8}$ \\
3 && 4 & $\ull{5}$ 		&& 5 & $\ull{6}$   && 6 & $\ull{8}$   && 7 & $\ull{9}$   && 8 & $\ull{11}$  && 9 & $\ull{12}$ \\
4 && 5 & $\ull{6}$ 		&& 6 & $\ull{8}$   && 7 & $\ull{10}$  && 8 & $\ull{12}$  && 9 & $\ull{14}$  && 10 & $\ull{16}$ \\
5 && 5 & $\ull{7}$ 		&& 6 & $\ull{9}$   && 7 & $\exc{11}$  && 8 & $\ull{14}$  && 10 & $\ull{16}$ && 11 & $\ull{19}$ \\
6 && 6 & $\exc{7}$		&& 7 & $\ull{10}$  && 8 & $\ull{13}$  && 9 & $\ull{16}$  && 10 & $\ull{19}$ && 11 & $\ull{22}$ \\
7 && 7 & $\ull{8}$ 		&& 8 & $\ull{11}$  && 9 & $\ull{14}$  && 9 & $\ull{18}$  && 11 & $\ull{21}$ && 12 & $\ull{24}$ \\
8 && 8 & $\ull{9}$ 		&& 9 & $\ull{12}$  && 9 & $\exc{15}$  && 10 & $\ull{19}$ && 11 & $\ull{23}$ && 12 & $\exc{26}$ \\
9 && $\ull{9}$ & $\ull{9}$ 	&& 9 & $\ull{13}$  && 10 & $\ull{17}$ && 11 & $\exc{20}$ && 12 & $\ull{25}$ && 13 & $\ull{29}$ \\
10 && $\ull{9}$ & $\ull{9}$	&& 10 & $\ull{13}$ && 11 & $\exc{17}$ && 12 & $\ull{22}$ && 13 & $\ull{26}$ && 14 & $\ull{31}$ \\
11 && $\ull{9}$ & $\ull{9}$	&& 11 & $\ull{14}$ && 12 & $\ull{18}$ && 13 & $\ull{23}$ && 14 & $\ull{28}$ && 15 & $\exc{32}$ \\
12 && $\ull{9}$ & $\ull{9}$	&& 12 & $\exc{14}$ && 13 & $\ull{19}$ && 14 & $\ull{24}$ && 15 & $\ull{29}$ && 15 & $\ull{34}$ \\
13 && $\ull{9}$ & $\ull{9}$	&& 13 & $\ull{15}$ && 14 & $\ull{20}$ && 14 & $\ull{25}$ && 15 & $\ull{30}$ && 16 & $\ull{36}$ \\
14 && $\ull{9}$ & $\ull{9}$	&& 14 & $\ull{15}$ && 14 & $\exc{20}$ && 15 & $\ull{26}$ && 16 & $\exc{31}$ && 17 & $\ull{37}$ \\
15 && $\ull{9}$ & $\ull{9}$	&& 14 & $\ull{16}$ && 15 & $\ull{21}$ && 16 & $\ull{27}$ && 17 & $\ull{33}$ && 18 & $\ull{39}$ \\
16 && $\ull{9}$ & $\ull{9}$	&& 14 & $\ull{16}$ && 16 & $\ull{22}$ && 17 & $\exc{27}$ && 18 & $\ull{34}$ && 19 & $\ull{40}$ \\
17 && $\ull{9}$ & $\ull{9}$	&& 14 & $\ull{16}$ && 17 & $\ull{22}$ && 18 & $\ull{28}$ && 19 & $\ull{35}$ && 20 & $\ull{41}$ \\
18 && $\ull{9}$ & $\ull{9}$	&& 14 & $\ull{16}$ && 18 & $\ull{23}$ && 19 & $\ull{29}$ && 20 & $\exc{35}$ && 20 & $\ull{42}$ \\
19 && $\ull{9}$ & $\ull{9}$	&& 14 & $\ull{16}$ && 19 & $\ull{23}$ && 20 & $\ull{30}$ && 20 & $\ull{36}$ && 21 & $\ull{43}$ \\
20 && $\ull{9}$ & $\ull{9}$	&& 14 & $\ull{16}$ && 20 & $\exc{23}$ && 20 & $\ull{30}$ && 21 & $\ull{37}$ && 22 & $\exc{44}$ \\
21 && $\ull{9}$ & $\ull{9}$	&& 14 & $\ull{16}$ && 21 & $\ull{24}$ && 21 & $\ull{31}$ && 22 & $\ull{38}$ && 23 & $\ull{45}$ \\
22 && $\ull{9}$ & $\ull{9}$	&& 14 & $\ull{16}$ && 21 & $\ull{24}$ && 22 & $\ull{31}$ && 23 & $\ull{39}$ && 24 & $\ull{46}$ \\
23 && $\ull{9}$ & $\ull{9}$	&& 14 & $\ull{16}$ && 21 & $\ull{25}$ && 23 & $\ull{32}$ && 24 & $\ull{39}$ && 25 & $\ull{47}$ \\
24 && $\ull{9}$ & $\ull{9}$	&& 14 & $\ull{16}$ && 21 & $\ull{25}$ && 24 & $\ull{32}$ && 25 & $\ull{40}$ && 26 & $\ull{48}$ \\
25 && $\ull{9}$ & $\ull{9}$	&& 14 & $\ull{16}$ && 21 & $\ull{25}$ && 25 & $\ull{33}$ && 26 & $\ull{41}$ && 26 & $\ull{49}$ \\
 26 && $\ull{9}$ & $\ull{9}$	&& 14 & $\ull{16}$ && 21 & $\ull{25}$ && 26 & $\ull{33}$ && 27 & $\ull{41}$ && 27 & $\ull{50}$ \\
 27 && $\ull{9}$ & $\ull{9}$	&& 14 & $\ull{16}$ && 21 & $\ull{25}$ && 27 & $\ull{33}$ && 27 & $\ull{42}$ && 28 & $\ull{50}$ \\
 28 && $\ull{9}$ & $\ull{9}$	&& 14 & $\ull{16}$ && 21 & $\ull{25}$ && 28 & $\ull{34}$ && 28 & $\ull{42}$ && 29 & $\ull{51}$ \\
 29 && $\ull{9}$ & $\ull{9}$	&& 14 & $\ull{16}$ && 21 & $\ull{25}$ && 29 & $\ull{34}$ && 29 & $\ull{43}$ && 30 & $\ull{52}$ \\
 30 && $\ull{9}$ & $\ull{9}$	&& 14 & $\ull{16}$ && 21 & $\ull{25}$ && 30 & $\exc{34}$ && 30 & $\ull{43}$ && 31 & $\ull{52}$ \\
 31 && $\ull{9}$ & $\ull{9}$	&& 14 & $\ull{16}$ && 21 & $\ull{25}$ && 30 & $\ull{35}$ && 31 & $\ull{44}$ && 32 & $\ull{53}$ \\
 32 && $\ull{9}$ & $\ull{9}$	&& 14 & $\ull{16}$ && 21 & $\ull{25}$ && 30 & $\ull{35}$ && 32 & $\ull{44}$ && 33 & $\exc{53}$ \\
 33 && $\ull{9}$ & $\ull{9}$	&& 14 & $\ull{16}$ && 21 & $\ull{25}$ && 30 & $\ull{35}$ && 33 & $\ull{44}$ && 34 & $\ull{54}$ \\
\bottomrule
\end{tabular}
\end{center}
\end{table}

\section{A sufficient condition for specific identifiability}\label{spec}
Assume we have a decomposition of a tensor $\tensor{A}$ as in \refeqn{eqn:cpd}. One could ask for an algorithm that detects whether this particular decomposition is unique, such as in Kruskal's lemma \cite{Kruskal1977COV}. In particular, one wonders if the algorithm we proposed is capable of giving a sufficient criterion to check if the decomposition is unique.

We begin with an observation concerning tensor subspaces. Assume that a tensor $\tensor{A} \in \F^{n_1 \times \cdots \times n_d}$, with $\F = \R$ or $\C$, whose decomposition is sought lives in a strict tensor subspace $A_1 \otimes \cdots \otimes A_d \subset \F^{n_1 \times \cdots \times n_d}$, with $\dim A_i = r_i \le n_i$ and at least one inequality strict. Letting $Q_i \in \F^{n_i \times r_i}$ be a basis for $A_i$, which, in practice, can be recovered using the HOSVD \cite{Tucker1966COV,Lathauwer2000aCOV,Vannieuwenhoven2012COV}, we may write
\[
\tensor{A} = (Q_1, \ldots, Q_d) \cdot \tensor{A}' \quad\mbox{with } \tensor{A}' \in \F^{r_1 \times \cdots \times r_d},
\] 
where the multilinear multiplication $\tensor{A} = (Q_1, \ldots, Q_d) \cdot \tensor{A}'$ can be defined as 
\[
 \tensor{A} = \sum_{j_1=1}^{r_1} \sum_{j_2=1}^{r_2} \cdots \sum_{j_d=1}^{r_d} \tensor{A}'_{j_1, j_2, \ldots, j_d} (Q_1 \vect{e}_{j_1}) \otimes (Q_2 \vect{e}_{j_2}) \otimes \cdots \otimes (Q_d \vect{e}_{j_d}),
\]
with $\vect{e}_{j_k}$ the $j_k$th standard basis vector of $\F^{r_k}$; see, e.g., \cite{Silva2008COV} for equivalent definitions. In the literature, the tuple $(r_1,r_2,\ldots,r_d)$ is called the multilinear rank of $\tensor{A}$ \cite{Hitchcock1927COV,Silva2008COV,Carlini2011COV}. The property of relevance to our discussion is the following.
\begin{theorem} \label{preservesuniqueness}
 Let $\tensor{A} \in A_1 \ktimes \cdots \ktimes A_d \subset \F^{n_1 \times \cdots \times n_d}$ be a tensor of rank $r$ where $A_i$ is a subspace of $\F^{n_i}$ of dimension $r_i \le n_i$. Let $Q_i$ be a matrix representing a basis for $A_i$. Then, $\tensor{A} = (Q_1, \ldots, Q_d) \cdot \tensor{A}'$ is $r$-identifiable if and only if $\tensor{A}' \in \F^{r_1 \times r_2 \times \cdots \times r_d}$ is $r$-identifiable.
\end{theorem}
\begin{proof}
Recall that the rank of $\tensor{A}$ and $\tensor{A}'$ are equal; see, e.g., \cite{Buczynski2013COV,Landsberg2012COV}.

If $\tensor{A}$ is $r$-identifiable, then $\tensor{A}'$ is also $r$-identifiable. Indeed, if we assume that $\tensor{A}'$ has two different decompositions, then $\tensor{A} = (Q_1, \ldots, Q_d) \cdot \tensor{A}'$ also has at least two different decompositions by the properties of multilinear multiplication.

Conversely, if $\tensor{A}'$ has a unique decomposition, it follows, from the previous argument, that $\tensor{A}$ has exactly one decomposition of the type
\begin{align*}
 \tensor{A} = \sum_{i=1}^r Q_1 \sten{x}{i}{1} \ktimes Q_2 \sten{x}{i}{2} \ktimes \cdots \ktimes Q_d \sten{x}{i}{d};
\end{align*}
any alternative decomposition 
\(
 \tensor{A} = \sum_{i=1}^r \sten{a}{i}{1} \ktimes \cdots \ktimes \sten{a}{i}{d}
\)
should thus have at least one $i$ and $k$ such that $\sten{a}{i}{k}$ is not contained in the span of $A_k$. This, however, immediately contradicts with \cite[Corollary 2.2]{Buczynski2013COV}: that corollary implies that the number of terms in such a decomposition of $\tensor{A}$ would have to be strictly larger than $r$.
\qquad\end{proof}

By combining this theorem with the existence of the sporadic cases where generic $r$-identifiability does not hold, see \refthm{main1}, e.g., $\Pj\C^{4} \times \Pj\C^{4} \times \Pj^{4}$ with $r = 6$, we can readily prove the existence of some specific tensors to which the algorithm proposed in \refsec{thealgorithm} cannot be applied straightforwardly.

\begin{example}[A problematic case] \label{ex_complications}
Consider, for instance, a general rank-$6$ tensor $\tensor{A}$ of multilinear rank $(4,4,4)$ in a space $\Pj\C^{n_1} \otimes \Pj\C^{n_2} \otimes \Pj\C^{n_3}$ that is generically $6$-identifiable. Note that several such spaces exist, as proved by \refalg{alg_generic_uniqueness}; however, let us consider spaces of the type $\Pj\C^n \times \Pj\C^n \times \Pj\C^n$ with $n \ge 6$, which are generically $6$-identifiable as mentioned in the introduction. Then, by \refthm{preservesuniqueness}, $\tensor{A}$ may be written as $\tensor{A} = (Q_1,Q_2,Q_3) \cdot \tensor{A}'$, $Q_i \in \C^{n \times 4}$, where $\tensor{A}'$ is a general tensor of rank $6$ in $\C^{4} \otimes \C^{4} \otimes \C^{4}$. From \cite[Section 5]{CO2012COV} it is known that a general $\tensor{A}'$ is not $6$-identifiable; indeed, it has two decompositions. On the other hand, by definition of generality, the subset of tensors that are not identifiable in an open neighborhood of $\tensor{A}$ has measure zero, otherwise generic $6$-identifiability of $\Pj\C^n \otimes \Pj\C^n \otimes \Pj\C^n$ would be contradicted. As a consequence, inspecting \refprop{infinitepoints}, it is clear that the same proof can be applied to this case,\footnote{It is not difficult to prove that since Terracini's lemma applies, i.e., the dimension of the span of the individual tangent spaces is maximal, for (general) $\tensor{A}'$ in $\Pj\C^4 \times \Pj\C^4 \times \Pj\C^4$, then it also applies for the particular $\tensor{A} = (Q_1,Q_2,Q_3)\cdot\tensor{A}'$ in $\Pj\C^n \times \Pj\C^n \times \Pj\C^n$.} so that the proposed algorithm will detect that the rank of the stacked Hessian is maximal. That is, the proposed algorithm correctly detects that only in a set of measure zero (in either the Zariski or Euclidean topology) the $6$-identifiability property fails around $\tensor{A}$, \emph{but the algorithm has no means to detect that $\tensor{A}$ is precisely in this set of measure zero}.

The reason of the above behavior can be geometrically understood as follows. 
Let $\Var{S}'$ be the Segre variety of $\Pj\C^{4} \ttimes \Pj\C^{4} \ttimes \Pj\C^{4}$, which is naturally 
embedded in the Segre variety $\Var{S}$ of $\Pj\C^{n} \ttimes \Pj\C^{n} \ttimes \Pj\C^{n}$, $n\ge6$. 
When we consider the {\it abstract secant variety} $A\sigma_6(\Var{S})$, {as defined in \cite{ChCi2006COV}, i.e., 
the Zariski closure in $\mathrm{Sym}^6(\Var{S})\times\Pj(\C^{64})$ of the variety of pairs $((p_1,\dots,p_6),p)$
 where $p\in\langle p_1,\dots,p_6\rangle$}, and the natural projection 
$\pi_6: A\sigma_6(\Var{S})\to \sigma_6(\Var{S})$, then the fibers of $\pi_6$ are singletons over general points 
of $\sigma_6(\Var{S})$, 
while over (general) points $p' \in \sigma_6(\Var{S}')$ they consist of a pair of points. Zariski's Main Theorem \cite[Corollary III.11.4]{Hartshorne1997COV} states that the inverse image of a normal point 
under a birational projective morphism is connected. Then we find that $\sigma_6(\Var{S}')$ is contained in 
the singular locus of $\sigma_6(\Var{S})$. Indeed, $\sigma_6(\Var{S})$ is two-folded at a general point of 
$\sigma_6(\Var{S}')$. The stacked Hessian only detects what happens in a neighbourhood of $p$ in one of the 
two folds: the behaviour is perfectly regular there. 
\end{example}

The previous example showed that a straightforward application of the Hessian criterion may fail if the given rank-$r$ decomposition corresponds to a singular point of $\sigma_r(\Var{S})$. We continue to show that this is the only type of failure that prevents us from applying the Hessian criterion. That is, if the given decomposition corresponds to a nonsingular point, then one can try to prove its identifiability using the Hessian criterion of \refprop{prop_reformulation}. To prove this, we introduce two preparatory lemmas.

\begin{lemma}\label{terraciniholds} Let $\Var{S} \subset \Pj\C^\Pi$ be
a Segre variety of dimension $\Sigma$. Let $p_1, p_2, \ldots, p_r \in
\Var{S}$ and $p \in \Sec{r}{\Var{S}}$ in the span $\langle p_1, p_2,
\ldots, p_r \rangle$. Assume that $p$ is not contained in the singular
locus of $\Sec{r}{\Var{S}}$ and assume that
$$
\dim( \langle \mathrm{T}_{p_1} \Var{S}, \mathrm{T}_{p_2} \Var{S},
\ldots, \mathrm{T}_{p_r} \Var{S} \rangle)= r(\Sigma+1)-1,
$$
which is the expected dimension of the secant variety. Then, the
variety $\Var{S}$ is not $r$-defective and the conclusion of
Terracini's lemma holds, i.e.,
$$
\mathrm{T}_p \Sec{r}{\Var{S}} = \langle \mathrm{T}_{p_1} \Var{S},
\mathrm{T}_{p_2} \Var{S}, \ldots, \mathrm{T}_{p_r} \Var{S} \rangle.
$$
\end{lemma}
\begin{proof} Our assumptions imply that the points $p_i$'s are
linearly independent. Thus the abstract secant variety
$A\sigma_r(\Var{S})$ is smooth at $(p,(p_1,\dots,p_r))$. The
projection map from the abstract secant variety to the secant variety
sends the tangent space to $A\sigma_r(\Var{S})$ at
$(p,(p_1,\dots,p_r))$ to the linear span $\langle \mathrm{T}_{p_1}
\Var{S}, \mathrm{T}_{p_2}\Var{S}, \ldots, \mathrm{T}_{p_r} \Var{S}
\rangle$, which is thus contained in the Zariski tangent space of
$\sigma_r(\Var{S})$ at $p$. Comparing the dimensions, the conclusion
follows, since $\dim(\sigma_r(\Var{S}))$ cannot be greater than
$r(\Sigma+1)-1$.  
\qquad\end{proof}

\begin{lemma}\label{normalsecant}
Let $\Var{S} \subset \Pj\C^\Pi$ be a Segre variety of dimension $\Sigma$. 
Let $p_1$, $p_2$, $\ldots$, $p_r$, $q\in\Var{S}$ be distinct points and let $p \in \Sec{r}{\Var{S}}$ be in the intersection of the spans $\langle p_1, p_2, \ldots, p_r\rangle \cap \langle q, p_2, \ldots, p_r \rangle$. Assume that 
the rank of the stacked Hessian $H$, defined in \refsec{sectionalgtesting}, at every {$p_1,\dots, p_r$} is maximal, i.e., equal to $\Sigma$,
then the point $p\in\Sec{r}{\Var{S}}$ is not normal, in particular it is singular.
\end{lemma}
\begin{proof} 
Consider the projection from the abstract secant variety $\pi\colon A\sigma_r(\Var{S})\to \sigma_r(\Var{S})$.
By the assumption on the rank of $H$, the \refprop{prop_reformulation} implies that $\Var{S}$ is generically $r$-identifiable. 
It follows that $\pi$ is a birational morphism. By assumption, after reordering the points, we have that, in affine notation,
$p=\sum_{i=1}^ra_ip_i=b_1q+\sum_{i=2}^rb_ip_i$ for convenient scalars $a_i$, $b_i$. 
Hence, the fiber $\pi^{-1}(p)$ contains the two points $(p,(p_1,p_2,\ldots, p_r))$ and $(p,(q,p_2,\ldots, p_r))$.
The connected component of the fiber passing through $(p,(p_1,p_2,\ldots, p_r))$ consists of just this single point,
 because it is contained in the $r$-contact locus of $T_p\sigma_r(\Var{S})$, which is zero-dimensional at $(p,(p_1,p_2,\ldots, p_r))$,
by the assumption on the rank of $H$. It follows from Zariski's Main Theorem \cite[Corollary III.11.4]{Hartshorne1997COV} that the point $p\in\Sec{r}{\Var{S}}$ is not normal.
\qquad\end{proof}

With the two previous lemmas, we get a criterion for detecting the uniqueness of a given decomposition of a tensor $p$, \emph{provided that we know that $p$ is not contained in the singular locus of the secant variety.} The criterion is the following.

\begin{theorem}\label{thm:specificuniqueness}
Let $p=\sum_{i=1}^r a_ip_i$ be a decomposition with $a_i\in\C$ and $p_i\in\Var{S}$, and assume
$p$ is a nonsingular point of $\Sec{r}{\Var{S}}$. Let $\mathrm{H} = \langle \Tang{p_1}{\Var{S}}, \Tang{p_2}{\Var{S}}, \ldots, \Tang{p_r} {\Var{S}} \rangle$. Then, the decomposition is unique if the rank of the stacked Hessian $H$, defined in \refsec{sectionalgtesting}, at every {$p_1,\dots, p_r$} is maximal, i.e., equal to $\Sigma$.
\end{theorem}
\begin{proof} We proceed similarly as in the proof of Proposition \ref{infinitepoints}: if $p$ is not $r$-identifiable, then we have, in affine notation,
\[
p = \sum_{i=1}^r a_i p_i = \sum_{i=1}^r b_i q_i
\]
with $a_i, b_i\in\C$, $q_i \in \Var{S}$. At least one of the $q_i \notin \{p_1,\ldots,p_r\}$; otherwise, $p$ would have two different expressions as a linear combination of the $p_i$, so that $p$ would be an element of the $(r-1)$-secant variety, and, hence, a singular point of the $r$-secant variety. By \reflem{terraciniholds},  Terracini's lemma applies, so that $\Tang{p}{\Sec{r}{\Var{S}}} = \Plane{H}$. Letting $b_i(t) \ne 0$ be a curve with a parameter $t$ in a neighborhood of $0$, in such a way that $b_i(0)=b_i$, the resulting tensor $p(t)= \sum_{i=1}^r b_i(t) q_i$ has a tangent space $\Tang{p(t)}{\Sec{r}{\Var{S}}}$ which is constant with respect to $t$ by Terracini's lemma because general points in the span of the $q_i$'s all have the same tangent space in the secant variety, and $p(t)$ moves in the span of the $q_i$'s. We can then choose $b_i(t)$ in such a way that $p(t) \notin \langle p_1,\ldots, p_r \rangle$, because otherwise we have $q_i \in \langle p_1, \ldots, p_r \rangle$, contradicting \reflem{normalsecant}. We may write
\[
p(t) = \sum_{i=1}^r a_i(t) p_i(t) \quad\mbox{with } a_i(0) = a_i \mbox{ and }
p_i(0) = p_i,
\]
where, by the previous argument, not all $p_i(t)$ can be constant. Then, we have infinitely many $p_1(t)$ such that $\Tang{p_1(t)}{\Var{S}} \subset \Tang{p(t)}{\Sec{r} {\Var{S}}} = \Plane{H}$.
\qquad\end{proof}

\begin{remark}[Modifications to \refalg{alg_generic_uniqueness}]
In light of \refthm{thm:specificuniqueness}, some minor modifications are required to make \refalg{alg_generic_uniqueness} work in the setting of specific identifiability, \emph{provided that we already know that the input rank-$r$ decomposition corresponds to a nonsingular point of the $r$-secant variety.} Step S1 may be removed; instead, each of the terms in the given rank-$r$ decomposition corresponds to one point $p_k \in \Var{S}$. Then, because the Hessian criterion must be checked for every point $p_k$, steps S6 and S7 should be repeated for every point. That is, for point $p_k$, the submatrices of the Hessian $H^l$ in \refeqn{eqn:one-hessian}, i.e., $(H_{IJ}^l)$ should be replaced with
\[
 (H_{IJ}^l)_{ij} = \frac{\partial^2 q_l}{\partial a_i^I \partial a_j^J} \biggr|_{p=p_k},\;\; i=1,\ldots,n_I,\; j=1,\ldots,n_J,\; 1\le I,J \le d.
\]
Note that $H^l \in \C^{\Sigma+d \times \Sigma+d}$ whose rank will, by definition, be less than $\Sigma$. Let $p_k = \vect{v}_1 \otimes \cdots \otimes \vect{v}_d$ in affine notation. One can verify through straightforward computations starting from \refeqn{eqn:polynomial-eqns} that, assuming $I < J$,
\[
H_{IJ}^l = ( \vect{v}_1^T, \ldots, \vect{v}_{I-1}^T, I, \vect{v}_{I+1}^T, \ldots, \vect{v}_{J-1}^T, I, \vect{v}_{J+1}^T, \ldots, \vect{v}_d^T ) \cdot \tensor{K},
\]
where $\tensor{K}_{i_1,\ldots,i_d} = k_{m(i_1,\ldots,i_d),l}$. For $J < I$, we have $H_{IJ}^l = (H_{JI}^l)^T$, and if $J=I$, then $H_{IJ} = 0$. If the rank of the stacked hessian $H = [\begin{smallmatrix} H^1 & \cdots & H^\ell \end{smallmatrix}]$ is maximal, i.e., equal to $\Sigma$, at $p_1, \ldots, p_r$, then we conclude that the Hessian criterion applies, and that the given decomposition is identifiable.
\end{remark}

In the next section, we give some sufficient conditions for the nonsingularity of a given tensor $\tensor{A}$ of small rank. Regarding this topic, we mention the results of \cite{BLCOV,BaCCOV,LandsbergOttaviani2013COV}, which solve the case of some \emph{symmetric} tensors of low rank; see Corollary 1.5 of \cite{BaCCOV}.

\section{Identifiability of specific tensors beyond Kruskal's bound} \label{sec_specific_identifiability}

In this section, we give examples how Theorem \ref{thm:specificuniqueness} can be implemented in some specific cases.
 This technique can be applied to all tensors of a given small rank, unless they belong to a set of measure zero
in the $r$-secant variety.
Since we know enough equations of the $r$-secant variety in a range that often is greater than Kruskal's range in (\ref{eq:kruskalrange}),
we may prove the uniqueness of a specific decomposition of a tensor, in cases where neither Kruskal's nor Domanov--De Lathauwer's criterion applies.
It is important to stress that this does not contradict Derksen's result in \cite{Derksen2013COV}, who proved that Kruskal's criterion is sharp for certain tensors in a set of measure zero.

\subsection{Some equations of secant varieties to Segre varieties}
We restrict ourselves to the case where the number of factors $d$ equals $3$:\footnote{At least for odd $d$ and small rank, one may expect that the technique presented in this subsection provides enough equations for applying \refthm{thm:specificuniqueness}.} let $ \tensor{A} \in \C^{n_1 \times{} n_2 \times{} n_3}$ with $n_1 \ge n_2 \ge n_3 \ge 2$. Recall that we can consider
\[
 \tensor{A} \in \C^{n_1 \times n_2 \times n_3} \simeq \C^{n_1} \otimes \C^{n_2} \otimes \C^{n_3} \simeq {\C^{n_1}}^* \otimes \C^{n_2} \otimes \C^{n_3} 
\]
using the identification of dual spaces $\C^{n_1} \simeq {\C^{n_1}}^*$.
Moreover, the last space can be identified with the space of maps
$\left( \C^{n_1} \to \C^{n_2} \otimes \C^{n_3}\right)$.
A well known technique, see, e.g.,
Chapter 7 in \cite{Landsberg2012COV}, to find some equations of  
$\Sec{r}{\Var{S}}$ is to compute the $(r+1)$-minors of the standard \emph{contraction map}
$$F_{ \tensor{A}}\colon\C^{n_1}\to \C^{n_2}\otimes\C^{n_3}.$$
The transpose of the matrix representing such map is usually called a \emph{flattening}, \emph{unfolding} or \emph{matricization}, and has size $n_1\times n_2 n_3$.
This technique gives nontrivial equations of $\Sec{r}{\Var{S}}$ only for $r<n_1$.

In order to have nontrivial equations of $\Sec{r}{\Var{S}}$ for larger values of $r$,
the following technique is useful. It was introduced in 
\cite{LandsbergOttaviani2013COV} in a geometric vector bundle setting. 
For every $p=1,\ldots, \lfloor\frac{n_3}{2}\rfloor$, we can consider the more general
contraction map\footnote{We consider the identification
$\wedge^p \C^{n_3}=\C^{n_3 \choose p}$.}
\[
A_{ \tensor{A}}\colon\C^{n_1}\otimes\C^{{n_3}\choose p}\to \C^{n_2}\otimes\C^{{n_3}\choose p+1},
\]
which is defined in the following way: if $\tensor{A}=\vect{a}_1\otimes \vect{a}_2\otimes \vect{a}_3$, then
\[
A_{\vect{a}_1\otimes \vect{a}_2\otimes \vect{a}_3}(\vect{f}\otimes \vect{g}):= (\vect{a}_1\cdot \vect{f})\vect{a}_2\otimes (\vect{g}\wedge \vect{a}_3), \quad \vect{f} \in \C^{n_1}, \vect{g} \in \C^{{n_3}\choose p}
\]
where $\vect{a} \cdot \vect{b}$ is the standard inner product, and in the general case it is defined by linearity;
that is, if
$\tensor{A} = \sum_{i=1}^{r} \sten{a}{i}{1} \ktimes \sten{a}{i}{2} \ktimes
\sten{a}{i}{3}$,
then
$$A_{ \tensor{A}}=\sum_{i=1}^{r} A_{\sten{a}{i}{1} \ktimes \sten{a}{i}{2} \ktimes
\sten{a}{i}{3}}.$$
The matrix of this more general contraction is sometimes called Young flattening.

For example, in the case $n_3=3$ with $p=1$, the matrix representing the linear map
$A_{ \tensor{A}}$ has size $3n_2\times 3n_1$ and, in convenient basis, it has the following block structure
$$\bgroup\begin{pmatrix}0&{{X}_{3}}&
     {-{X}_{2}}\\
     -{X}_{3}&
     0&
     {{X}_{1}}\\
          {X}_{2}&
     -{X}_{1}&0\\
     \end{pmatrix}\egroup,$$
where $X_i$, $i=1,2,3$, are the three $n_2\times n_1$ slices of 
$\tensor{A}$. As another example, consider the case $n_3=4$ with $p=2$. Then, the matrix representing the linear map
$A_{ \tensor{A}}$ has size $4n_2\times 6n_1$ and has the following block structure
$$\bgroup\begin{pmatrix}{
      -{X}_{2}}&
     {-{X}_{3}}&
     0&
     {-{X}_{4}}&
     0&
     0\\
     {X}_{1}&
     0&
     {-{X}_{3}}&
     0&
     {-{X}_{4}}&
     0\\
     0&
     {X}_{1}&
     {X}_{2}&
     0&
     0&
     {-{X}_{4}}\\
     0&
     0&
     0&
     {X}_{1}&
     {X}_{2}&
     {X}_{3}\\
     \end{pmatrix}\egroup, $$
where $X_i$, $i=1,\ldots, 4$, are the four $n_2\times n_1$ slices of 
$\tensor{A}$.

We have $\mathrm{rk\ }A_{ \vect{a}_1\otimes \vect{a}_2\otimes \vect{a}_3}={{n_3-1}\choose p}$.
If $\tensor{A} = \sum_{i=1}^{r} \sten{a}{i}{1} \ktimes \sten{a}{i}{2} \ktimes
\sten{a}{i}{3}$ it follows
$\mathrm{rk\ }A_{ \tensor{A}}\le r{{n_3-1}\choose p}$, so that the minors of size
$r{{n_3-1}\choose p}+1$ of $A_{ \tensor{A}}$ vanish on $ \tensor{A}\in\Sec{r}{\Var{S}}$, hence furnishing some of the latter's equations.
The celebrated Strassen equations introduced in \cite{Strassen1983COV}
correspond to the particular case $n_1=n_2$, $n_3=3$, $p=1$.

It is important to compute the tangent space at a determinantal locus. The direct computation from minors is computationally infeasible. The following Lemma makes the computation much easier.

\begin{lemma}\label{smoothlemma}
Let  $\tensor{A}_0=\sum_{i=1}^{r} \sten{a}{i}{1} \ktimes \sten{a}{i}{2} \ktimes
\sten{a}{i}{3}\in \C^{n_1 \times{} n_2 \times{} n_3}$, choose $1\le p\le\lfloor\frac{n_3}{2}\rfloor$, and let
$A_{ \tensor{A}_0}\colon\C^{n_1}\otimes\C^{{n_3}\choose p}\to \C^{n_2}\otimes\C^{{n_3}\choose {p+1}}$ be the corresponding contraction maps.
Consider $\ker A_{ \tensor{A}_0}\subset\C^{n_1}\otimes\C^{{n_3}\choose p}$ and
$\left(\mathrm{Im\ }A_{ \tensor{A}_0}\right)^\perp\subset\C^{n_2}\otimes\C^{{n_3}\choose {p+1}}$.
If the dimension of the image of 
\begin{equation}\label{eq:normalspace}\ker A_{ \tensor{A}_0}\otimes \left(\mathrm{Im\ }A_{ \tensor{A}_0}\right)^\perp\to \C^{n_1 \times{} n_2 \times{} n_3}
\end{equation}
is equal to the codimension of $\Sec{r}{\Var{S}}$, then the tensor $\tensor{A}_0$ is a smooth point of $\Sec{r}{\Var{S}}$.
\end{lemma}
\begin{proof}
Notice that in the formulation we have used the identification of 
$\C^{{n_3}\choose p}$ with $\wedge^p\C^{n_3}$ and of
$\C^{{n_3}\choose {p+1}}$ with $(\wedge^{p+1}{\C^{n_3}})^*$,
which is the dual space of $\wedge^{p+1}\C^{n_3}$, and exploited $(\wedge^{p+1}\C^{n_3})^* \otimes \wedge^{p}\C^{n_3} \to {\C^{n_3}}^*\simeq\C^{n_3}$.
Now the proof follows from Theorem 8.4.2 of \cite{LandsbergOttaviani2013COV}. Indeed, the image of 
(\ref{eq:normalspace}) coincides with the conormal space 
at ${ \tensor{A}_0}$ of the variety cut by minors of size $r{{n_3-1}\choose p}+1$ of $A_{ \tensor{A}}$,
for general $\tensor{A}\in \C^{n_1 \times{} n_2 \times{} n_3}$,
so it has the same dimension as the normal space of the variety cut by these minors.
\qquad\end{proof}

\reflem{smoothlemma} is the basic tool we use in this section,
in order to apply our identifiability algorithm to a specific tensor $p$.
It provides a sufficient condition that $p$ corresponds to a nonsingular point,
which is requisite for applying Theorem \ref{thm:specificuniqueness}.

\begin{example} In the case $n_3=4$, $p=2$, we have seen that the matrix representing the linear map
$A_{ \tensor{A}_0}$ has size $4n_2\times 6n_1$; it can be written as a matrix $A'$ of size
$4\times 6$, with entries linear in the coordinates of $\C^4=\C^{n_3}$. We have a kernel of dimension $6n_1-3r$, whose basis gives a matrix $K$ of size
$6n_1\times   (6n_1-3r)$, which can be written as a matrix $K'$
of size $6\times   (6n_1-3r)$, with entries linear in the coordinates of $\C^{n_1}$. Correspondingly we have $\left(\mathrm{Im\ }A_{ \tensor{A}_0}\right)^\perp$ of dimension
$4n_2-3r$, whose basis gives a (transposed) matrix $M$ of size $(4n_2-3r)\times 4n_2$.
We get a matrix $M'$ of size $(4n_2-3r)\times 4$, with entries linear in the coordinates of $\C^{n_2}$.
The multiplication $M'\cdot A'\cdot K'$ has size $(4n_2-3r)\times (6n_1-3r)$
and its entries, treating the coordinates of $\tensor{A}$ as indeterminates, define cartesian equations for the image of the map in (\ref{eq:normalspace}).
\end{example}

The following proposition reveals some cases where the zero locus of these equations contains $\Sec{r}{\Var{S}}$ as irreducible component.

\begin{proposition}\label{caseq}
Let $p=\lfloor\frac{n_3}{2}\rfloor$. The 
variety $$\left\{\tensor{A}\in\Pj\C^{\Pi} \;\bigl|\; \textrm{\ the minors of size\ } r\tbinom{n_3-1}{p}+1\textrm{\ of\ }A_{ \tensor{A}}\textrm{\ vanish\ } \right\}$$
contains $\Sec{r}{\Var{S}}$ as irreducible component, 
if $n_1$, $n_2$, $n_3$, and $r$ appear in the ``Proposed'' column in \reftab{tab_specific_identifiability_range}.
Thus, if $r$ satisfies the above inequalities, then \reflem{smoothlemma} applies to all tensors of border rank $r$
not in some indeterminate subset of measure zero.
\end{proposition}
\begin{proof}
In every case we can pick a random point in $\Sec{r}{\Var{S}}$ and compute the tangent space at that point of the zero locus of the minors of size $r{{n_1-1}\choose p}+1$ of $A_{ \tensor{A}}$,
according to \refeqn{eq:normalspace}. The dimension of this tangent space coincides, in every case, with the dimension of  $\Sec{r}{\Var{S}}$.
\qquad\end{proof}

\begin{remark}
Conversely, when $r$ does not satisfy the inequalities in \refprop{caseq}, the assumption on the dimension
of image of \refeqn{eq:normalspace} is never satisfied and \reflem{smoothlemma} does not apply.
We notice that Theorem 1.2 in \cite{Landsberg2012augustCOV} provides, in the cubic case $n_1=n_2=n_3$, a lower bound on the rank of $A_{ \tensor{A}}$
for general  $\tensor{A}$, which grows asymptotically as $2n_1$.
\end{remark}

It is instructive to compare the range in which specific identifiability can be checked using the criterion of Kruskal, in \refeqn{eq:kruskalrange}, the criterion of Domanov--De Lathauwer \cite{Domanov2013bCOV}, and the method proposed in this paper; this is presented in \reftab{tab_specific_identifiability_range}.

\begin{table}[tb] \footnotesize
\caption{Upper bounds on the rank $r$ for which specific identifiability of a rank-$r$ decomposition can be verified with the proposed criterion, Kruskal's criterion in \refeqn{eq:kruskalrange}, and Domanov--De Lathauwer's criterion in \cite{Domanov2013bCOV}. Indicated in bold face is the criterion with the widest range.}
\label{tab_specific_identifiability_range}
\begin{center}
\begin{tabular}{cccc}
\toprule
$(n_1,n_2,n_3)$ & Proposed & Kruskal & $\substack{\mbox{Domanov--} \\ \mbox{De Lathauwer}}$ \\
\midrule
$(4,4,4)$ & $r \le 4$ & $r \le \mathbf{5}$ & $r \le \mathbf{5}$ \\
$(5,5,5)$ & $r \le \mathbf{7}$  & $r \le 6$  & $r \le 6$  \\
$(6,6,6)$ & $r \le \mathbf{8}$  & $r \le \mathbf{8}$  & $r \le \mathbf{8}$  \\
$(7,7,7)$ & $r \le \mathbf{11}$ & $r \le 9$  & $r \le 9$  \\
$(8,8,8)$ & $r \le \mathbf{12}$ & $r \le 11$ & $r \le 11$ \\
$(9,9,9)$ & $r \le \mathbf{15}$ & $r \le 12$ & $r \le 13$ \\
\bottomrule
\end{tabular}
\end{center}
\end{table}

In fact, the upper bound for $(9,9,9)$ in \reftab{tab_specific_identifiability_range} can be improved slightly by generalizing \reflem{smoothlemma}.

\begin{lemma}
Let  $\tensor{A}=\sum_{i=1}^{r} \sten{a}{i}{1} \ktimes \sten{a}{i}{2} \ktimes
\sten{a}{i}{3}\in \C^{n_1 \times{} n_2 \times{} n_3}$, choose $1\le p_i\le\lfloor\frac{n_i}{2}\rfloor$, and let
$A^1_{ \tensor{A}}\colon\C^{n_1}\otimes\C^{{n_3}\choose {p_3}}\to \C^{n_2}\otimes\C^{{n_3}\choose {p_3+1}}$,
$A^2_{ \tensor{A}}\colon\C^{n_2}\otimes\C^{{n_1}\choose {p_1}}\to \C^{n_3}\otimes\C^{{n_1}\choose {p_1+1}}$ ,
$A^3_{ \tensor{A}}\colon\C^{n_3}\otimes\C^{{n_2}\choose {p_2}}\to \C^{n_1}\otimes\C^{{n_1}\choose {p_2+1}}$  be the corresponding contraction maps.
Consider $\ker A^i_{ \tensor{A}}\subset\C^{n_1}\otimes\C^{{n_3}\choose p}$ and
$\left(\mathrm{Im\ }A_{ \tensor{A}}\right)^\perp\subset\C^{n_2}\otimes\C^{{n_3}\choose {p+1}}$.
If the dimension of the image of 
\begin{equation}\label{eq:normalspace2}\bigoplus_{i=1}^3\ker A^i_{ \tensor{A}}\otimes \left(\mathrm{Im\ }A^i_{ \tensor{A}}\right)^\perp\to \C^{n_1 \times{} n_2 \times{} n_3}
\end{equation}
is equal to the codimension of $\Sec{r}{\Var{S}}$, then the tensor $\tensor{A}$ is a smooth point of $\Sec{r}{\Var{S}}$.
\end{lemma}
\begin{proof}
It is a straightforward extension of Lemma
\ref{smoothlemma}.
\qquad\end{proof}

The following Proposition generalizes slightly Proposition \ref{caseq}.
\begin{proposition}\label{caseq2}
Let $\Var{S} = \Pj\C^{9}
\ttimes 
\Pj\C^{9} \ttimes \Pj\C^{9}$ embedded in $\Pj\C^{729}$. The common zero locus of the minors of size $r{{8}\choose 4}+1=70r+1$ of $A^i_{ \tensor{A}}$, for $i=1, 2, 3$,
 contains $\Sec{r}{\Var{S}}$ as irreducible component for $r\le 16$.
\end{proposition}
\begin{proof}
We can pick a random point in $\Sec{r}{\Var{S}}$ and compute the tangent space at that point of the common zero locus of the minors of size 
$70r+1$ of $A^i_{ \tensor{A}}$, according to (\ref{eq:normalspace2}). The codimension of this tangent space is $329$, which coincides with the codimension of  $\Sec{r}{\Var{S}}$.
We remark that, in this case, the codimension of the tangent spaces of the zero locus of the minors of size 
$70r+1$ of each individual $A^i_{ \tensor{A}}$ is $196$. By intersecting two individual tangent spaces (for example for $i=1, 2$), we get a linear subspace which
already has the desired
codimension $329$.
\qquad\end{proof}

\subsection{The algorithm at work for a specific tensor} \label{sec_specific_example}
Exploiting the equations for the $r$-secant variety presented in the previous subsection, we can now apply the algorithm for specific identifiability to some particular cases. Let $\tensor{A} = \sum_{i=1}^{r} \sten{a}{i}{1} \ktimes \sten{a}{i}{2} \ktimes \sten{a}{i}{3}$ a given decomposition for $(n_1, n_2, n_3)$, $r$, in a case appearing in \refprop{caseq}. Then, we can hope to apply our criterion for specific identifiability.

\begin{example}
We consider the following rank $7$ tensor $\tensor{A} \in \C^5\otimes\C^5\otimes\C^5$:
\begin{align}\label{555r7}
\tensor{A} = \begin{bmatrix} 1 \\ 1 \\ 1 \\ 1 \\ 1 \end{bmatrix} \otimes \begin{bmatrix} 1 \\ 2 \\3 \\4 \\5 \end{bmatrix} \otimes  \begin{bmatrix} 1 \\ 5 \\ 7 \\ -5 \\ -7 \end{bmatrix} 
\;+\; 
\begin{bmatrix} 4 \\ 3 \\ 2 \\ -1 \\ -2 \end{bmatrix} \otimes \begin{bmatrix} 11 \\ 13 \\ 12 \\ 15 \\ 14 \end{bmatrix} \otimes \begin{bmatrix} -2 \\ 6 \\ 5 \\ -3 \\ 6 \end{bmatrix} 
\;+\;
\sum_{i=1}^5 \vect{e}_i \otimes \vect{e}_i \otimes \vect{e}_i,
\end{align} 
with $\vect{e}_i$ the $i$th standard basis vector in $\C^5$.
This example can be studied neither with Kruskal's criterion nor with Domanov--De Lathauwer's condition, as we learn from \reftab{tab_specific_identifiability_range}. 
We show that the decomposition \refeqn{555r7} is unique.
Let $\Var{S}=\Pj\C^5\times\Pj\C^5\times\Pj\C^5$.
We compute the map $A_\tensor{A}\colon\C^5\otimes\wedge^2\C^5\to\C^5\otimes\wedge^3\C^5$
which has rank $42$.
Hence the subspaces $\ker A_\tensor{A}$ and $\left(\mathrm{Im\ }A_{\tensor{A}}\right)^\perp$
have both dimension $8$.
We compute the image of $\left(\ker A_\tensor{A}\right)\otimes \left(\mathrm{Im\ }A_{\tensor{A}}\right)^\perp$ in $\C^5\otimes\C^5\otimes\C^5$, which has codimension $34$,
this image is the normal space to $\Sec{7}{X}$ at the point corresponding to $\tensor{A}$.
It follows that $\Sec{7}{X}$ is smooth at the point corresponding to $\tensor{A}$. In the ancillary files, we included a Macaulay2 script for verifying this computation.
Then, we may apply \refalg{alg_generic_uniqueness} with the only change that step S1 is replaced by
the decomposition \refeqn{555r7} and $r=7$. The algorithm runs, getting the matrix
$T$ in step S2 of size $ 125\times 105$. The null space matrix $K^T$ of step S4 has size
$34\times 91$. Note that $l=\ell=34$. Steps S6 and S7b should be performed for each of the seven points. In step S6, we construct $34$ Hessian matrices of size $12\times 12$. In step S7b, the stacked Hessian $H$ has size $12\times 408$.
Its rank is $12$, for each of the seven points, hence concluding the proof.
\end{example}

\section{Conclusions} \label{sec_conclusions}
We presented a sufficient condition for generic $r$-identifiability along with an algorithm verifying it. Using this algorithm, we proved that in all spaces of dimension less than $15000$, except for the known exceptions, tensors of subgeneric rank are generically $r$-identifiable. Thereafter, we extended the sufficient condition to the case of specific $r$-identifiability, and demonstrated that our algorithm still works, provided that the specific rank-$r$ decomposition can be shown to correspond to a nonsingular point of the $r$-secant variety. Using some local equations for this variety, we were able to prove the identifiability of a specific tensor, whose identifiability could not be investigated using the criterions of Kruskal and Domanov--De Lathauwer. 

The contribution of this work is twofold: first, we showed that in spaces of practical size generic $r$-identifiability holds, so that a ``random'' tensor in such spaces admits a unique rank decomposition. Second, a novel promising direction for investigating specific identifiability was presented: the proposed criterion can, in principle, verify specific identifiability up to the optimal rank value, provided that a good test for nonsingularity of points on secant varieties of Segre varieties can be designed.

Unfortunately, little is known about the singularities of secant varieties of Segre varieties. As a consequence, our results concerning specific identifiability currently can only slightly improve the range of feasible cases with respect to Kruskal's and Domanov--De Lathauwer's conditions. However, the approach outlined here can, in contrast, be applied up to the optimal rank value, and will benefit from advances made in the characterization of equations and singularities of the $r$-secant variety. This study is, nevertheless, well beyond the scope of this paper, and will require advances in the state-of-the-art in algebraic geometry.

 \section*{Acknowledgements}
 The first two authors are members of italian GNSAGA-INDAM.
 The third author was supported by a Ph.D.~Fellowship of the Research Foun\-da\-tion--Flanders (FWO).

\bibliographystyle{amsplain}
\bibliography{shortstrings,IDFT}

\vskip 0.5cm

\footnotesize {\bf Authors' addresses:}

\noindent Luca Chiantini, Dipartimento di Ingegneria dell'Informazione e Scienze Matematiche,\\ Universit\`a di Siena, Italy, {\tt luca.chiantini@unisi.it}
\vskip 0.15cm
\noindent Giorgio Ottaviani, Dipartimento di Matematica e Informatica ``U. Dini'',\\Universit\`a di Firenze,
Italy, {\tt ottavian@math.unifi.it}
\vskip 0.15cm
\noindent Nick Vannieuwenhoven,  Department of Computer Science,\\ KU Leuven, Leuven, Belgium,
 {\tt nick.vannieuwenhoven@cs.kuleuven.be}

\end{document}